\documentclass[a4paper]{article}

\usepackage[french,english]{babel}
\usepackage[T1]{fontenc}
\usepackage[utf8]{inputenc}
\usepackage{amsmath,amssymb,amsfonts,amsthm}
\usepackage{makeidx}
\usepackage{color}
\usepackage{empheq}
\usepackage{cancel}
\usepackage[top=20mm, bottom=20mm, left=20mm, right=20mm]{geometry}
\usepackage{stmaryrd} % pour \llbracket et \rrbracket.
\usepackage{enumerate}
\usepackage{hyperref} % à utiliser avec \texorpdfstring{$\mathcal{L}(\gamma)$}{Lg} dans les titres.
\usepackage{booktabs}
\usepackage[all]{xy} % pour les diagrammes commutatifs.

\usepackage[vcentermath,noautoscale]{youngtab}
\everymath{\Yboxdim10pt}
\everydisplay{\Yboxdim10pt}

\SetSymbolFont{stmry}{bold}{U}{stmry}{m}{n} % formule magique pour qu'il ne râle pas avec \boldmath.

\newtheorem{thm}{Theorem}[section]%%%%%%%%%%% Titre en gras droit, corps de texte italique %%%%%%%%%%%%
\newtheorem{conj}{Conjecture}
\newtheorem{conjbis}{Conjecture}
\newtheorem{prop}[thm]{Proposition}
\newtheorem{cor}[thm]{Corollary}
\newtheorem{lem}[thm]{Lemma}

\theoremstyle{remark}%%%%%%%%%%%%%%%%%%%%%%%%% Titre en italique, corps de texte droit %%%%%%%%%%%%%%%%%
\newtheorem{rem}{\textsc{Remark}}

\theoremstyle{definition}%%%%%%%%%%%%%%%%%%%%% Titre en gras droit, corps de texte droit %%%%%%%%%%%%%%%
\newtheorem{defi}[thm]{Definition}

\newcommand{\Cb}{\mathbb{C}}
\newcommand{\Cc}{\mathcal{C}}

\newcommand{\Ib}{\mathbb{I}}
\newcommand{\Id}{\mathrm{Id}}

\newcommand{\Kc}{\mathcal{K}}

\newcommand{\Mc}{\mathcal{M}}

\newcommand{\Nb}{\mathbb{N}}

\newcommand{\Rb}{\mathbb{R}}

\newcommand{\Sc}{\mathcal{S}}

\newcommand{\Xc}{\mathcal{X}}

\DeclareMathOperator{\com}{com}
\DeclareMathOperator{\Diag}{Diag}

\newcommand{\drond}{\partial}
\renewcommand{\vec}[1]{\overrightarrow{#1}}

\title{\textsf{\textbf{On the concavity of a sum of elementary symmetric polynomials}}}
\author{\textsf{Xavier Lachaume}}
\date{\textsf{\normalsize{Summer 2017}}}

\makeindex
\begin{document}
\maketitle

\begin{center}
\textsf{
Laboratoire de Mathématiques et de Physique Théorique \\
Université de Tours - UMR 7350 du CNRS \\
Parc de Grandmont - 37200 Tours - France
}

\texttt{xavier.lachaume@lmpt.univ-tours.fr}
\end{center}

\vfill

\textsc{Abstract:}
We introduce a new problem on the elementary symmetric polynomials $\sigma_k$, stemming from the constraint equations of some modified gravity theory. For which coefficients is a linear combination of $\sigma_k$ $1/p$-concave, with $0 \leq k \leq p$? We establish connections between the $1/p$-concavity and the real-rootedness of some polynomials built on the coefficients.

We conjecture that if the restriction of the linear combination to the positive diagonal is a real-rooted polynomial, then the linear combination is $1/p$-concave. Using the theory of hyperbolic polynomials, we show that this would be implied by a short algebraic statement: if the polynomials $P$ and $Q$ of degree $n$ are real-rooted, then $\sum_{k=0}^n P^{(k)}Q^{(n-k)}$ is real-rooted as well. This is not proven yet.

We conjecture more generally that the global $1/p$-concavity is equivalent to the $1/p$-concavity on the positive diagonal. We prove all our guessings for $p=2$. The way is open for further developments.

\vfill

\tableofcontents

\newpage

\section{Introduction} \label{sectionintro}
\subsection{Physical and geometrical context}

Since its discovery one century ago, the theory of General Relativity (GR) has exhibited numerous qualities, as much in its theoretical properties as in its consistency with astronomical observations. The recent detections of gravitational waves seal the concordance between this theory and almost all that can be measured with current astronomical tools. However, some questions still remain open: the explanation of what is called dark matter and dark energy to account for huge but precise discrepancy between the matter and energy computed from the observations of the universe and those expected by GR; and the theoretical compatibility with quantum field theory, which would give birth to a quantum theory of gravitation.

Seeking to fill those gaps and, more generally, to better understand GR, theoretical physicists explored many ways to modify GR in order to test the properties of the newly created theories. This is the family of modified theories of gravity. Among them, the Lovelock theories appeared in the 1970's when D. Lovelock studied what would be a generalisation of GR in more than 4 dimensions (see \cite{Love70}, \cite{Love71}). The Lovelock theories are still quite unexplored in their mathematical properties.

In \cite{Lac17-2}, we try to solve the constraint equations of those theories: they are the necessary conditions for a space-like manifold to be the initial data of some Lovelock space-time. They are equations on the geometry of a manifold. They are known since the end of the 1980's (see \cite{Tei87}, \cite{Cho88}, \cite{Cho09-692}), but never studied from a geometrical point of view, unlike the constraint equations of GR. When the search for a solution is restricted to a conformal class, it can be reduced to a non-linear PDE system. A specific case of this problem -- a conformally flat vacuum space-time -- is equivalent to the prescription of a linear combination of $\sigma_k$-curvatures on the manifold. The prescription of a single $\sigma_k$-curvature is a classical geometrical problem called $\sigma_k$-Yamabe problem, a generalisation of the Yamabe curvature prescription problem. It has been studied in the 2000's, see \cite{Via00}, \cite{Via01}, \cite{Guan03}\cite{Li03OnSome}, \cite{Li05}, \cite{Li03AFully}, \cite{Guan07}, \cite{Gursky07}, \cite{Sheng07} \ldots\ In all these works, a very fundamental property of the $\sigma_k$ function is invoked: it is $\frac{1}{k}$-concave.

Here is our problem: when is a linear combination of $\sigma_k$'s concave? This question is only about the coefficients of the linear combination. It is not about gravitation, theoretical physics, geometry or PDE: it is a simple algebraic question, to which we sought algebraic answers. This is the topic of this paper.

In the first section we introduce the notations and some quick lemmas about concavity. Insofar as concavity of algebraic expressions is connected with the real-rootedness of polynomials, we dedicated the second section to this subject. Then in the third section we formulate some conjectures about the sought-after set of coefficients. In the fourth section we show connections with the theory of hyperbolic polynomials of L. G\aa rding which give birth to totally algebraic conjectures more likely to be solved.

We proved our conjectures for a linear combination of $1$, $\sigma_1$, $\sigma_2$.

\subsection{Notations}
\begin{defi}
Let $n \geq 1$, and $\Omega$ an open set of $\Rb^n$. For $f, g \in \Cc^2(\Omega \rightarrow \Rb)$, $P, Q \in \Rb[X]$, $x = (x_1, \ldots, x_n) \in \Rb^n$, $0 \leq k \leq n$, we write
\[\begin{array}{rll}
\toprule
f_i &= \drond_{x_i} f, \\
f_{ij} &= \drond_{x_i}\drond_{x_j} f, \\
\midrule
D(f) &= (f_i)_{1 \leq i \leq n} \in \Cc^1(\Omega \rightarrow \Rb^n) &\text{ the gradient of } f, \\
H(f) &= (f_{ij})_{1 \leq i,j \leq n} \in \Cc^0(\Omega \rightarrow \Sc_n(\Rb)) &\text{ the hessian of } f, \\
\midrule
g &\equiv f &\text{ if $f$ and $g$ have the same sign on $\Omega$}, \\
P &\equiv Q &\text{ if there exists $R \in \Rb(X)$ such that $P = R^2 Q$}, \\
\midrule
\sigma_k(x) &= \displaystyle{\sum_{i_1 < \ldots < i_k} x_{i_1}\ldots x_{i_k}} &\text{ the $k$-th elementary symmetric polynomial}. \\
\bottomrule
\end{array}\]

We assimilate real polynomials and their polynomial function on $\Rb$. We define the following sets:
\[\Gamma_k = \left\{y \in \Rb^n \quad \big| \quad \sigma_i(y) > 0 \quad \forall \ 0 \leq i \leq k\right\}.\]
Let us notice that $\Gamma_n = \left(\Rb_+^*\right)^n$. We write $\Ib = (1, \ldots, 1) \in \Rb^n$, and keep the same notation whatever the dimension is when there is no ambiguity. We denote by $\Delta$ the open diagonal half-axis
\[\Delta = \Rb_+^* \cdot \Ib = \left\{(t, \ldots, t) \in \Gamma_n \ | \ t > 0\right\}.\]

For $x, y \in \Rb^n$, we denote by
\[\begin{array}{rl}
\toprule
 ^t\!x & \text{ the transpose of } x, \\
(x|y) = ^t\!\!xy = ^t\!\!yx = \sum_i x_i y_i & \text{ the scalar product of $x$ and $y$}, \\
x \otimes y = (x_iy_j)_{ij} \in \Mc_n(\Rb) & \text{ the tensor product of $x$ and $y$}, \\
x'_i = (x_1, \ldots, \cancel{x_i}, \ldots, x_n) \in \Rb^{n-1} & \text{ the vector $x$ without its $i$-th coordinate}, \\
\midrule
(e_i)_{1 \leq i \leq n} & \text{ the canonic basis of } \Rb^n. \\
\bottomrule
\end{array}\]

For $\vec{a} = (a_0, a_1, \ldots, a_p) \in \Rb^{p+1}$, we set
\[\begin{array}{rll}
\toprule
f_{\vec{a}} &= a_0 + a_1 \sigma_1 + \ldots + a_p \sigma_p &\in \Rb[X_1, \ldots, X_n], \\
P_{\vec{a}} &= a_0 + a_1 X + \ldots + a_p X^p &\in \Rb[X], \\
\bottomrule
\end{array}\]
we then define
\[\begin{array}{rlll}
\toprule
\bar{f}_{\vec{a}} &= f(X, \ldots, X) &= \displaystyle{a_0 \binom{n}{0} + a_1 \binom{n}{1} X + \ldots + a_p \binom{n}{p} X^p} &\in \Rb[X], \\
\tilde{P}_{\vec{a}} &= X^p P\left(\dfrac{1}{X}\right) &= \displaystyle{a_p + a_{p-1}X + \ldots + a_0 X^p} &\in \Rb[X]. \\
\bottomrule
\end{array}\]
\end{defi}
We omit the subscript in case there is no ambiguity.

Finally, let us introduce three sets which are to be determined in the current paper.
\[\begin{array}{rl}
\toprule
\Kc_n^p &= \displaystyle{\left\{\vec{a} \in \Rb^{p+1} \quad \big| \quad f_{\vec{a}}^{1/p} \text{ is concave on } \Gamma_n\right\}} \\
\Xc_n^p &= \displaystyle{\left\{\vec{a} \in \Rb^{p+1} \quad \big| \quad \bar{f}_{\vec{a}}^{1/p} \text{ is concave on } \Rb_+^*\right\}} \\
\Xi_n^p &= \displaystyle{\left\{\vec{a} \in \Rb^{p+1} \quad \big| \quad \bar{f}_{\vec{a}}^{\phantom{1/p}} \text{ is real-rooted}\right\}} \\
\bottomrule
\end{array}\]

Obviously, $\Kc_n^p \subset \Xc_n^p$. Is the reverse true? Here is the first of our two main conjectures: the concavity of $f_{\vec{a}}$ on the diagonal $\Delta$ is a necessary and sufficient condition to its concavity on all $\Gamma_n$.
\begin{conj} \label{mainconj1}
Let $p \in \Nb^*$, $\vec{a} = (a_0, a_1, \ldots, a_p) \in (\Rb_+)^{p+1}$. Then
\begin{align*}
f_{\vec{a}} \text{ is $\frac{1}{p}$-concave on } \Gamma_n \quad &\Longleftrightarrow \quad f_{\vec{a}} \text{ is $\frac{1}{p}$-concave on } \Delta \\
&\Longleftrightarrow \quad \bar{f}_{\vec{a}} \text{ is $\frac{1}{p}$-concave on } \Rb_+^*.
\end{align*}
That is to say,
\[\Kc_n^p = \Xc_n^p.\]
\end{conj}
We still do not have a proof of this conjecture, yet we can prove the equivalence in a few cases which will be presented in section \ref{sectionsigma}.

A weaker version of this conjecture can be expressed in terms of real-rootedness of $\bar{f}_{\vec{a}}$.
\begin{conj} \label{mainconj2}
Let $p \in \Nb^*$, $\vec{a} = (a_0, a_1, \ldots, a_p) \in (\Rb_+)^{p+1}$. Then
\[f_{\vec{a}} \text{ is $\frac{1}{p}$-concave on } \Gamma_n \quad \Longleftarrow \quad \bar{f}_{\vec{a}} \text{ is real-rooted}.\]
That is to say,
\[\Xi_n^p \subset \Kc_n^p.\]
\end{conj}

Indeed, as will be shown in proposition \ref{concroots}, $\Xi_n^p \subset \Xc_n^p$.

\subsection{Concavity lemmas}
\begin{defi}
Let $\Omega$ be an open domain of $\Rb^n$, $u$ be a function of $\Cc^0(\Omega \rightarrow \Rb_+^*)$, and $\mu > 0$. We say that $u$ is $\mu$-concave if $u^\mu$ is concave.
\end{defi}

Let us present two lemmas.

\begin{lem} \label{concmu}
Let $\Omega$ be an open domain of $\Rb^n$, $u$ be a function of $\Cc^2(\Omega \rightarrow \Rb_+^*)$, and $\mu > 0$. Then
\[H\left(u^\mu\right) \equiv u H(u) + (\mu-1) D(u) \otimes D(u).\]
$u$ being $\Cc^2$, we get the following equivalence:
\[u \text{ is $\mu$-concave on } \Omega \qquad \Longleftrightarrow \qquad u H(u) + (\mu-1) D(u) \otimes D(u) \text{ is negative on } \Omega.\]
\end{lem}
\begin{proof}
This comes directly from the computation of the hessian:
\begin{equation}
H\left(u^\mu\right) = \mu u^{\mu-2} \left[u H(u) + (\mu-1) D(u) \otimes D(u)\right].
\end{equation}
\end{proof}

\begin{lem} \label{concgamma}
Let $\Omega$ be an open domain of $\Rb^n$, $a, b \in \Cc^2(\Omega \rightarrow \Rb_+^*)$, and $\alpha, \beta > 0$ such that $a$ is $\alpha$-concave and $b$ is $\beta$-concave.

Then $(ab)$ is $\gamma$-concave, with \[\dfrac{1}{\gamma} := \dfrac{1}{\alpha} + \dfrac{1}{\beta}.\]
\end{lem}
\begin{proof}
The functions are $\Cc^2$, so they are concave if and only if their hessian is negative. Using the lemma \ref{concmu}, we compute the hessian of $(ab)^\gamma$:
\begin{align*}
H\left((ab)^\gamma\right)
	&\equiv a b H(ab) + (\gamma-1) D(ab) \otimes D(ab) \\
	&= ab\left[H(a)b+D(a)\otimes D(b) + D(b)\otimes D(a) + a H(b)\right] \\
	&\qquad + (\gamma-1)\left(aD(b) + D(a)b\right)\otimes \left(aD(b) + D(a)b\right) \\
	&= b^2\left[aH(a) + (\alpha-1) D(a)\otimes D(a)\right] + (\gamma - \alpha)b^2 D(a)\otimes D(a) \\
	&\qquad + a^2\left[bH(b) + (\beta-1) D(b)\otimes D(b)\right] + (\gamma - \beta)a^2 D(b)\otimes D(b) \\
	&\qquad + ab \gamma \left[D(a) \otimes D(b) + D(b) \otimes D(a)\right] \\
	&= b^2\left[aH(a) + (\alpha-1) D(a)\otimes D(a)\right] + a^2\left[bH(b) + (\beta-1) D(b)\otimes D(b)\right] \\
	&\qquad - \dfrac{1}{\alpha + \beta} \left(b\alpha D(a) - a \beta D(b)\right) \otimes \left(b\alpha D(a) - a \beta D(b)\right)
\end{align*}
which is the sum of three negative matrices.
\end{proof}

\begin{rem}
This short lemma, which did not seem to be formulated in this way in the literature, is analogous to Hölder's inequality. It enables to skip the use of Hölder's inequality, such as in the proof of the theorem \ref{concsigmak+lsigmak} in the subsection \ref{othertools}.
\end{rem}

Before getting to the heart of matter, we present some properties of real polynomials which will be useful for next section. Except the proposition \ref{propmanu} and the theorem \ref{thmdeg123}, those are already known results.

\section{Polynomials} \label{sectionpolynomials}
\subsection{Concavity and real-rootedness}

\begin{defi}
Let $P \in \Rb_+[X]$ be a polynomial with non-negative coefficients, $\deg P = n$.

We say that $P$ verifies
\begin{enumerate}[$(P1)$]
	\item if and only if $P$ is real-rooted (thus with non-positive roots); \label{P1}
	\item if and only if $PP'' + \left(\frac{1}{n}-1\right)P'^2 \leq 0$ on $\Rb$; \label{P2}
	\item if and only if $P$ is $\frac{1}{n}$-concave on $\Rb_+^*$. \label{P3}
\end{enumerate}
\end{defi}

\begin{prop} \label{concroots}
\[(P\ref{P1}) \qquad \Longrightarrow \qquad (P\ref{P2}) \qquad \Longrightarrow \qquad (P\ref{P3}).\]
\end{prop}
\begin{proof}
The second implication comes directly from lemma \ref{concmu}.

For the first implication, here is a classical proof that uses Cauchy-Schwarz inequality. Let us suppose that $P = c(X+\mu_1)\ldots(X+\mu_n)$. At first we compute that
\[\frac{P'}{P} = \sum_{k=1}^n \frac{1}{X+\mu_k},\]
and
\[\left(\frac{P'}{P}\right)' = \frac{PP'' - P'^2}{P^2} = -\sum_{k=1}^n \frac{1}{(X+\mu_k)^2}.\]
Then,
\begin{align*}
P P'' + \left(\frac{1}{n}-1\right)P'^2
  &\equiv n\frac{PP'' - P'^2}{P^2} + \left(\frac{P'}{P}\right)^2 \\
  &= -n\sum_{k=1}^n \frac{1}{(X+\mu_k)^2} + \left(\sum_{k=1}^n \frac{1}{X+\mu_k}\right)^2 \\
  &\leq 0 \quad \text{ on } \Rb,
\end{align*}
according to Cauchy-Schwarz inequality, thus $(P\ref{P2})$.
\end{proof}
\begin{proof}
There is an other way to prove directly $(P\ref{P1}) \Longrightarrow (P\ref{P3})$, by using our lemma \ref{concgamma}.

There exist $c, \mu_1, \ldots, \mu_n \in \Rb_+$ such that 
\[P = c(X+\mu_1)\ldots(X+\mu_n).\]
We proceed by induction on the degree of $P$.
\begin{itemize}
\item For $n=1$, $P = c(X + \mu_1)$ is an affine function so it is concave on $\Rb_+^*$.
\item For $n \geq 2$, the polynomial $c(X + \mu_1)\ldots(X + \mu_{n-1})$ is real-rooted, so by hypothesis we assume that
\[\left[c(X + \mu_1)\ldots(X + \mu_{n-1})\right]^{1/n-1} \text{ is concave on } \Rb_+^*.\]
The function $X + \mu_n$ is concave on $\Rb_+^*$, so we can apply the lemma \ref{concgamma} and
\[\left(c(X + \mu_1)\ldots(X + \mu_{n-1})(X + \mu_n)\right)^{\gamma} \text{ is concave},\]
with $\gamma = \frac{1}{\frac{1}{1/(n-1)} + \frac{1}{1}} = \frac{1}{n}$.
\end{itemize}
\end{proof}

It seems to us that $(P\ref{P2}) \Longrightarrow (P\ref{P1})$, but we were not able to prove it. We hope that some reader will find the answer to this conjecture.
\begin{conj} \label{conj3}
Let $P \in \Rb_+[X]$ be a polynomial with non-negative coefficients, $\deg P = n$. Then
\[P \text{ is real-rooted} \qquad \Longleftarrow \qquad PP'' + \left(\dfrac{1}{n}-1\right)P'^2 \leq 0 \text{ on } \Rb.\]
\end{conj}

One of the hints for this conjecture is the following result.
\begin{prop} \label{propmanu}
Let $P \in \Rb[X]$, $\deg P = n$, and $\tilde{P} := X^n P\left(\dfrac{1}{X}\right)$. Then
\[(P^{1/n})'' \quad \equiv \quad (\tilde{P}^{1/n})''.\]
This implies that $P$ is $\frac{1}{n}$-concave on $\Rb_+^*$ if and only if $\tilde{P}$ is.
\end{prop}
\begin{proof}
It is a simple computation:
\begin{align*}
\tilde{P}' &= nX^{n-1}P \left(\frac{1}{X}\right) - X^{n-2}P'\left(\frac{1}{X}\right), \\
\tilde{P}'' &= n(n-1)X^{n-2}P \left(\frac{1}{X}\right) - 2(n-1)X^{n-3} P'\left(\frac{1}{X}\right) + X^{n-4}P''\left(\frac{1}{X}\right),
\end{align*}
so
\[\tilde{P}\tilde{P}'' + \left(\frac{1}{n}-1\right) \tilde{P}'^2 = X^{2(n-2)}\left[PP'' + \left(\frac{1}{n}-1\right)P'^2\right].\]
\end{proof}
We know that $P$ is real-rooted if and only if $\tilde{P}$ is, so it would not be surprising that both equivalences of the $\frac{1}{n}$-concavity and the real-rootedness are related to each other.

However, we can show this conjecture for $\deg P \leq 3$.
\begin{thm} \label{thmdeg123}
Let $P \in \Rb_+[X]$, $\deg P \leq 2$. Then
\[(P\ref{P1}) \qquad \Longleftrightarrow \qquad (P\ref{P2}) \qquad \Longleftrightarrow \qquad (P\ref{P3}).\]

Let $P \in \Rb_+[X]$, $\deg P = 3$. Then
\[(P\ref{P1}) \qquad \Longleftrightarrow \qquad (P\ref{P2}).\]
\end{thm}
\begin{proof}

\begin{itemize}
\item For $n=1$, $(P\ref{P1})$, $(P\ref{P2})$ and $(P\ref{P3})$ are always true.
\item For $n=2$, let $P = a_0 + a_1 X + a_2 X^2$, and $\Delta_P := a_1^2 - 4a_0 a_2$ its discriminant. Then
\begin{align*}
P' &= a_1 + 2 a_2 X, \\
P'' &= 2 a_2, \\
PP'' + \left(\frac{1}{2}-1\right)P'^2 &= - \frac{\Delta_P}{2}.
\end{align*}
Hence
\[(P\ref{P1}) \quad \Longleftrightarrow \quad \Delta_P \geq 0 \quad \Longleftrightarrow \quad (P\ref{P2}) \quad \Longleftrightarrow \quad (P\ref{P3}).\]
\item For $n=3$, let $P = a_0 + a_1 X + a_2 X^2 + a_3 X^3$. We apply the Cardan method, setting
\begin{align*}
Y &:= X + \dfrac{a_2}{3a_3}, \\
Q &:= Y^3 + pY + q,
\end{align*}
with
\[p := \frac{a_1}{a_3} - \frac{a_2^2}{3a_3^2}, \qquad \qquad q := \frac{a_0}{a_3} - \frac{a_1a_2}{3a_3^2} + \frac{2a_2^3}{27a_3^3},\]
and
\begin{align*}
\Delta_Q &:= -(4p^3 + 27 q^2) \\
&= a_1^2a_2^2 + 18a_0a_1a_2a_3 - 27a_0^2a_3^2 - 4a_1^3a_3 - 4a_2^3a_0
\end{align*}
the discriminant of $Q$.

Then $P(X) = a_3 Q\left(X + \dfrac{a_2}{3a_3}\right)$, so $P$ is real-rooted if and only if $Q$ is. Thus
\[(P\ref{P1}) \quad \Longleftrightarrow \quad (Q\ref{P1}) \quad \Longleftrightarrow \quad \Delta_Q \geq 0.\]
Moreover,
\[\left[PP'' + \left(\frac{1}{3}-1\right) P'^2\right](X) = a_3^2\left[QQ'' + \left(\frac{1}{3}-1\right) Q'^2\right]\left(X + \dfrac{a_2}{3a_3}\right),\]
so
\[(P\ref{P2}) \quad \Longleftrightarrow \quad (Q\ref{P2}).\]

Let us compute
\begin{align*}
Q' &= 3Y^2 + p, \\
Q'' &= 6Y, \\
QQ'' + \left(\frac{1}{3}-1\right) Q'^2 &= 2pY^2 + 6qY - \frac{2}{3}p^2.
\end{align*}
We compute the discriminant of this polynomial of degree 2:
\[\delta_Q = 36q^2 + \frac{16}{3}p^3 = -\frac{4}{3}\Delta_Q.\]
Let us suppose that $(Q\ref{P2})$ is true, ie. $2pY^2 + 6qY - \frac{2}{3}p^2 \leq 0$ on $\Rb$. Necessarily, $p \leq 0$ and $\delta_Q \leq 0$, so $\Delta_Q \geq 0$, and $(Q\ref{P1})$.
\end{itemize}
\end{proof}

\begin{rem}
The reverse implication $(P\ref{P3}) \Longrightarrow (P\ref{P2})$ is not true anymore from $\deg P = 3$ on. Indeed, let
\[P = X^3 + X^2 + \dfrac{1}{3}X.\]
Then $P = X\left(X^2 + X + \dfrac{1}{3}\right)$ is not real-rooted. However,
\[PP'' + \left(\dfrac{1}{3}-1\right) P'^2 = -\dfrac{2}{27}(3X+1).\]
Hence $\left(P^{1/3}\right)'' \equiv -\dfrac{2}{27}(3X+1) < 0$ on $\Rb_+^*$.
\end{rem}

\subsection{Real-rooted polynomials}

\subsubsection{Number of real roots}
In order to prove that $P$ is real-rooted, we used the convenient property of polynomials of small degrees: their real-rootedness is equivalent to the positivity of one single quantity, the discriminant. This simple tool cannot be used anymore for $\deg P \geq 4$. A discriminant can always be defined, and its positivity is a necessary condition for the polynomial to be real-rooted, but not sufficient anymore. The discriminant is an algebraic expression in the coefficients of the polynomial, which is positive when the polynomial is real-rooted with simple roots, vanishes when two roots are equal, and changes its sign when a pair of real roots becomes strictly complex or conversely. So the sign of the discriminant only gives the number of real roots modulo 4.

The complete determination of the number of real roots needs more complex discrimination systems; see for instance \cite{Yang96}, \cite{Yang97}, \cite{Liang99}.

Before those computational theorems, there were classical theorems to determine the number of real roots of a polynomial. Sturm's theorem (1829), one of the oldest, gives the number of roots in a given real interval, using analytic arguments. Here we present two others with no interval restriction and algebraic tools.

\subsubsection{Algebraic characterisations}

\begin{thm}[1968, \cite{Coste03}, \cite{Dieu80}]
Let $P \in \Rb[X]$, $\deg P = n$, and
\[L(P,X,Y) := \dfrac{P(X)P'(Y) - P(Y)P'(X)}{X-Y} \in \Rb[X,Y]\]
a symmetric polynomial with coefficients
\[L(P,X,Y) =: \sum_{i,j = 1}^{n}l_{ij}X^{i-1}Y^{j-1}.\]
We introduce the quadratic form, applied to a vector $v=(v_1,\ldots,v_n)$,
\[Q(P,v) := \sum_{i,j = 1}^{n}l_{ij}v_iv_j.\]
Let $(s,t)$ be the signature of $Q$. Then the number of distinct real roots of $P$ is $s-t$, and the number of distinct complex roots is $s+t$.

Hence $P$ is real-rooted if and only the signature of $Q$ is $(s,0)$.
\end{thm}

In our particular case, all the coefficients of $P$ are supposed to be non-negative, hence all the real roots have to be non-positive. With this restriction, an older theorem holds.

\begin{defi}
A sequence $(a_0, a_1, a_2, \ldots) \in \Rb^\Nb$ is said to be \emph{totally positive} (or a \emph{P\'olya frequency sequence}) if and only if all the minors of the infinite matrix
\[\begin{pmatrix}
a_0 & 0 & 0 & \ldots \\
a_1 & a_0 & 0 & \ldots \\
a_2 & a_1 & a_0 & \ldots \\
\vdots & \vdots & \vdots & \ddots
\end{pmatrix}\]
are non-negative.
\end{defi}

\begin{thm}[1951, \cite{Aissen51}]
Let $P = a_0 + a_1 X + \ldots + a_n X^n \in \Rb_+[X]$, with $a_0 \neq 0$. Then $P$ is real-rooted if and only if the sequence $(a_0, a_1, \ldots, a_n, 0, 0, \ldots)$ is totally positive.
\end{thm}

\subsubsection{Log-concavity}
Now we present a necessary criterion for a polynomial to be real-rooted. It is based on Newton's inequalities.
\begin{defi}
A sequence $(a_0, a_1, a_2 \ldots) \in \Rb^\Nb$ is said to be \emph{log-concave} if and only if, for all $1 \leq k$,
\[a_{k-1} a_{k+1} \leq a_k^2.\]

A polynomial $P = a_0 + a_1 X + \ldots + a_n X^n \in \Rb[X]$ is said to be log-concave if and only if the sequence $(a_0, a_1, \ldots, a_n, 0, 0, \ldots)$ is log-concave.
\end{defi}

\begin{thm}[1707, Newton]
Let $\mu = (\mu_1, \mu_2, \ldots, \mu_n) \in \Rb_+^n$, and $\sigma_k := \sigma_k(\mu)$ for $0 \leq k \leq n$. Then for $1 \leq k \leq n-1$,
\[\dfrac{\sigma_{k-1}}{\binom{n}{k-1}} \dfrac{\sigma_{k+1}}{\binom{n}{k+1}} \leq \left(\dfrac{\sigma_k}{\binom{n}{k}}\right)^2.\]
This implies the weaker inequality:
\[\sigma_{k-1}\sigma_{k+1} \leq \left(\sigma_k\right)^2.\]
\end{thm}

\begin{cor} \label{corlogconc}
Let $P \in \Rb_+[X]$. If $P$ is real-rooted, then $P$ is log-concave.
\end{cor}

\subsubsection{Kurtz criterion}
Log-concavity however is not sufficient for a polynomial to be real-rooted. But some criteria are sufficient conditions, such as Kurtz's criterion which is a sort of stronger log-concavity.

\begin{thm}[1992, \cite{Kurtz92}]
Let $P = a_0 + a_1 X + \ldots + a_n X^n \in \Rb_+[X]$. If for all $1 \leq k \leq n-1$
\[4a_{k-1}a_{k+1} < a_k^2,\]
then all the roots of $P$ are real and distinct.
\end{thm}

On the other hand, there is no weaker log-concavity-like condition of the same form that works.
\begin{thm}[1992, \cite{Kurtz92}]
For all $\varepsilon > 0$, and $n \leq 2$, there is a polynomial $P = a_0 + a_1 X + \ldots + a_n X^n \in \Rb_+[X]$ which satisfies
\[(4-\varepsilon) a_{k-1}a_{k+1} < a_k^2\]
for all $1 \leq k \leq n-1$ and which has some non-real roots.
\end{thm}

Now we can introduce our main problem.

\section{Elementary symmetric polynomials} \label{sectionsigma}

As we explained in introduction, our purpose is to investigate the concavity of the functions $f_{\vec{a}}$. The example of one single $\sigma_k^\alpha, \alpha > 0$, teaches us that the function is ``less concave'' on the diagonal $\Delta$. To make it more precise, let us recall a result made explicit by M. Marcus and L. Lopes.
\begin{thm}[1957, \cite{Marcus57}] \label{concsigmak1surk}
\[\sigma_k \text{ is $\frac{1}{k}$-concave on } \Gamma_k.\]
\end{thm}
According to lemma \ref{concmu},
\[H\left(\sigma_k^\alpha\right) \equiv \sigma_k^{1/k} H\left(\sigma_k^{1/k}\right) + \left(k\alpha - 1\right) D\left(\sigma_k^{1/k}\right) \otimes D\left(\sigma_k^{1/k}\right),\]
which is negative when $\alpha < 1/k$. Conversely, if $\alpha > 1/k$, $\sigma_k^\alpha$ cannot be concave on $\Delta$ where it reduces to $t \mapsto \binom{n}{k}^\alpha t^{k\alpha}$.

We suppose that this result still holds for a general $f_{\vec{a}}$.

\setcounter{conjbis}{0}
\begin{conjbis}
Let $p \in \Nb^*$, $\vec{a} = (a_0, a_1, \ldots, a_p) \in (\Rb_+)^{p+1}$. Then
\begin{align*}
f_{\vec{a}} \text{ is $\frac{1}{p}$-concave on } \Gamma_n \quad &\Longleftrightarrow \quad f_{\vec{a}} \text{ is $\frac{1}{p}$-concave on } \Delta \\
&\Longleftrightarrow \quad \bar{f}_{\vec{a}} \text{ is $\frac{1}{p}$-concave on } \Rb_+^*.
\end{align*}
That is to say,
\[\Kc_n^p = \Xc_n^p.\]
\end{conjbis}

For now we only proved it for particular cases.

\subsection{With a determinant}
One way to study the concavity of $f_{\vec{a}}^{1/p}$ is to compute the determinant of its hessian. In this paragraph we shall do it for $p=2$ and prove conjecture \ref{mainconj1} in this case.

First we need a computational lemma:
\begin{lem} \label{detHfalpha}
Let $\Omega$ be an open set of $\Rb^n$, $f \in \Cc^2(\Omega \rightarrow \Rb_+^*)$, $\alpha > 0$. Then
\begin{equation}
\det H(f^\alpha) \equiv f \det H + (\alpha-1) ^t\!D \left(\com H\right) D. \label{Hgalpha}
\end{equation}
\end{lem}
\begin{proof}
We know that
\[H(f^\alpha) \equiv f H(f) + \left(\alpha - 1\right) D(f) \otimes D(f).\]
We set
\[H := H(f) = \left(\begin{array}{l|c|r} & & \\ H_1 & \ldots & H_n \\ & & \end{array}\right), \qquad
H_j := \begin{pmatrix} f_{1j} \\ \vdots \\ f_{nj} \end{pmatrix}, \qquad D := D(f) = \begin{pmatrix} f_1 \\ \vdots \\ f_n \end{pmatrix}.\]
Then
\begin{align*}
\det H(f^\alpha) &\equiv \det \left[f H + (\alpha-1) D \otimes D\right] \\
&= \det\left[f H_1 + (\alpha-1) f_1 D \ | \ \ldots \ | \ f H_n + (\alpha-1) f_n D\right] \\
&= f^n \det H + f^{n-1}(\alpha-1)\sum_{j=1}^n \det\left[H_1 \ | \ \ldots \ | \ f_j D \ | \ \ldots \ | \ H_n\right].
\end{align*}
We have to determine each of the $n$ terms of the sum, by expanding along the $j$-th row:
\begin{align*}
\det\left[H_1 \ | \ \ldots \ | \ f_j D \ | \ \ldots \ | \ H_n\right]
	&= f_j \sum_{i=1}^n (-1)^{i+j} f_i \det \begin{pmatrix} f_{11} & \ldots & \cancel{f_{1j}} & \ldots & f_{1n} \\ \vdots & & \vdots & & \vdots \\ \cancel{f_{i1}} & \ldots & \cancel{f_{ij}} & \ldots & \cancel{f_{in}} \\ \vdots & & \vdots & & \vdots \\ f_{n1} & \ldots & \cancel{f_{nj}} & \ldots & f_{nn} \end{pmatrix} \\
	&= f_j \sum_{i=1}^n (-1)^{i+j}f_i\bar{H}_{ij},
\end{align*}
where $\bar{H}_{ij}$ is the minor of $H$. Hence
\begin{align*}
\sum_{j=1}^n \det\left[H_1 \ | \ \ldots \ | \ f_j D \ | \ \ldots \ | \ H_n\right]
	&= \sum_{j=1}^n \sum_{i=1}^n f_j f_i (-1)^{i+j}\bar{H}_{ij} \\
	&= ^t\!\!D \left(\com H\right) D,
\end{align*}
which gives the claimed expression.
\end{proof}

This formula gives us the following result.
\begin{thm} \label{thmp2}
Let $\vec{a} = (a_0, a_1, a_2) \in (\Rb_+)^3$. We set $f_{\vec{a}} := a_0 + a_1 \sigma_1 + a_2 \sigma_2 \in \Rb_+[X_1,\ldots,X_n]$, and
\[\bar{f}_{\vec{a}}(X) := f_{\vec{a}}(X,\ldots,X) = a_0 \binom{n}{0} + a_1 \binom{n}{1} X + a_2 \binom{n}{2} X^2 \in \Rb_+[X].\]
Then
\[\vec{a} \in \Kc_n^2 \qquad \Longleftrightarrow \qquad \bar{f}_{\vec{a}} \text{ is real-rooted}.\]
In other words,
\[\sqrt{a_0 + a_1 \sigma_1 + a_2 \sigma_2} \text{ is concave on } \Gamma_n \qquad \Longleftrightarrow \qquad n a_1^2 - 2(n-1)a_0 a_2 \geq 0.\]
\end{thm}
\begin{rem}
Because of theorem \ref{thmdeg123}, this shows conjecture \ref{mainconj1} for $p=2$.
\end{rem}

\begin{proof}
If $a_2 = 0$, $f_{\vec{a}}$ is affine and all the properties are true.

Let us suppose that $a_2 > 0$. We use lemma \ref{detHfalpha} and take $f = f_{\vec{a}}$, and $\Omega = \Gamma_n$. Then
\begin{align*}
D(x) &= \begin{pmatrix} a_1 + a_2(\sigma_1(x) - x_1) \\ \vdots \\ a_1 + a_2(\sigma_1(x) - x_n) \end{pmatrix} = (a_1 + a_2 \sigma_1(x))\Ib - a_2 x, \\
H(x) &= a_2 \begin{pmatrix} 0 & 1 & \ldots & 1 \\ 1 & 0 & & \vdots \\ \vdots & & \ddots & 1 \\ 1 & \ldots & 1 & 0 \end{pmatrix}
\end{align*}
and
\[\det H(f^\alpha) \equiv f \det H + (\alpha - 1)^tD (\com H) D.\]

It is a classical result that $H(x)$ can be diagonalised in an orthonormal basis into
\[a_2\begin{pmatrix} n-1 & 0 & \ldots & 0 \\ 0 & -1 & & \vdots \\ \vdots & & \ddots & 0 \\ 0 & \ldots & 0 & -1 \end{pmatrix}.\]
Hence $\det H(x) = a_2^n(-1)^{n-1}(n-1)$. Moreover, one can check that
\[H(x)^{-1} = \dfrac{1}{(n-1)a_2}\left(\Ib \otimes \Ib - (n-1)\Id\right).\]
So
\begin{align*}
^tD (\com H) D &= ^tD (^t \com H) D \\
&= ^tD (\det H) H^{-1} D \\
&= ^tD\left[(-a_2)^{n-1}\left(\Ib^t\Ib - (n-1)\Id\right)\right]D \\
&= (-a_2)^{n-1}\left[\left(^t\Ib D\right)^2 - (n-1)^tDD\right] \\
&= (-a_2)^{n-1}\left[\left(\sum_{i=1}^n f_i\right)^2 - (n-1) \sum_{i=1}^n f_i^2\right].
\end{align*}

Then
\[\det H(f^\alpha) \equiv (-a_2)^{n-1}\left((n-1)a_2f + (\alpha-1)\left[\left(\sum_{i=1}^nf_i\right)^2 - (n-1)\sum_{i=1}^nf_i^2\right]\right).\]
Yet $f_i(x) = a_1 + a_2 \sigma_1(x) - a_2 x_i$, so
\begin{align*}
\sum_{i=1}^n f_i(x)^2 &= a_1^2 + (n-1)(a_1 + a_2 \sigma_1(x))^2 - 2a_2^2 \sigma_2(x), \\
\left(\sum_{i=1}^nf_i(x)\right)^2 &= \left(na_1 + (n-1)a_2\sigma_1(x)\right)^2 \\
	&= n^2 a_1^2 + 2n(n-1)a_1a_2\sigma_1(x) + (n-1)^2 a_2^2\sigma_1(x)^2, \\
\left(\sum_{i=1}^nf_i(x)\right)^2 - (n-1)\sum_{i=1}^nf_i(x)^2 &= na_1^2 + 2(n-1)a_1a_2\sigma_1(x) + 2(n-1)a_2^2\sigma_2(x).
\end{align*}
Finally, if we take $\alpha = 1/2$,
\begin{align*}
\det H(f^{1/2})(x) &\equiv (-a_2)^{n-1}\bigg((n-1)a_2(a_0 + a_1 \sigma_1(x) + a_2\sigma_2(x)) \\
& \qquad - \frac{1}{2}\left[na_1^2 + 2(n-1)a_1a_2\sigma_1(x) + 2(n-1)a_2^2\sigma_2(x)\right]\bigg) \\
&= \frac{1}{2}(-a_2)^{n-1}\left[2(n-1)a_0a_2 - n a_1^2\right].
\end{align*}

Hence, if $na_1^2 -2(n-1)a_0a_2 > 0$, $\det H(f^{1/2})$ has a constant sign, which is the parity of $n$. Its eigenvalues do not vanish on $\Gamma_n$, which is a connected set, so the concavity of $f^{1/2}$ is constant. For $t>0$,
\begin{align*}
H(f^{1/2})(t\cdot\Ib) &\equiv f(t)H(f)(t) - \frac{1}{2}D(f)(t) \otimes D(f)(t) \\
	&= \left(a_0 + a_1 nt + a_2 \frac{n(n-1)}{2}t^2\right) a_2H(t\cdot\Ib) - \frac{1}{2}(a_1 + a_2(n-1)t)^2 \Ib \otimes \Ib \\
	&\underset{t\rightarrow\infty}{\sim} \frac{a_2^2}{2}(n-1)\left[n(\Ib \otimes \Ib - \Id) - (n-1)\Ib \otimes \Ib\right]t^2 \\
	&\equiv \Ib \otimes \Ib - n\Id.
\end{align*}
Let $u=(u_1,\ldots,u_n) \in \Rb^n$. Let us compute
\begin{align*}
^t\!u(\Ib \otimes \Ib - n\Id)u &= (u|\Ib)^2 - (\Ib|\Ib)(u|u) \\
&\leq 0
\end{align*}
according to Cauchy-Schwarz inequality. Hence $H(f^{1/2})$ is a negative matrix at $t\cdot\Ib$ for $t$ large enough, thus negative everywhere on $\Gamma_n$.

Conversely, if $na_1^2 -2(n-1)a_0a_2 < 0$, $\det H(f^{1/2})$ has a constant sign, which is opposite of the parity of $n$. The eigenvalues of $H(f^{1/2})$ cannot be all negative, so $f^{1/2}$ cannot be concave on $\Gamma_n$.

We can handle the case $na_1^2 - 2(n-1) a_0 a_2 = 0$ because the set of concave functions is closed.
\end{proof}

\subsection{Other tools} \label{othertools}

\subsubsection{\texorpdfstring{$\sigma_{k+1}/\sigma_k$}{sk+1/sk}}
In this paragraph we shall exhibit a peculiar case of $\vec{a} \in \Kc_n^p$, using a concavity result from M. Marcus and L. Lopes:
\begin{thm}[1957, \cite{Marcus57}] \label{concsigmak+1sigmak}
Let $0 \leq k \leq n-1$. Then $\dfrac{\sigma_{k+1}}{\sigma_k}$ is concave on $\Gamma_n$.
\end{thm}
This result had been generalised to couples $(k, k+l)$ by J. B. MacLeod and P. Bullen, M. Marcus:
\begin{thm}[1959, \cite{Leod59}, 1961, \cite{Bullen61}] \label{concsigmak+lsigmak}
Let $0 \leq l \leq n$, $0 \leq k \leq n-l$. Then $\dfrac{\sigma_{k+l}}{\sigma_k}$ is $\frac{1}{l}$-concave on $\Gamma_n$.
\end{thm}
\begin{proof}
The original proof in \cite{Leod59} is quite involved. The proof of \cite{Bullen61} is shorter and invokes Hölder's inequality. We present here an even shorter proof that uses our lemma \ref{concgamma}:
\[\dfrac{\sigma_{k+l}}{\sigma_k} = \dfrac{\sigma_{k+l}}{\sigma_{k+l-1}}\dfrac{\sigma_{k+l-1}}{\sigma_{k+l-2}}\ldots\dfrac{\sigma_{k+1}}{\sigma_k},\]
and each of the $l$ factors of this product is $1$-concave, so the product is $\frac{1}{l}$-concave.
\end{proof}

Nevertheless we shall only use the first theorem, to deduce:
\begin{thm} \label{akak+1}
Let $0 \leq k \leq n-1$, $a_{k}, a_{k+1} \in \Rb_+^*$. Then
\[a_k \sigma_k + a_{k+1} \sigma_{k+1} \text{ is $\frac{1}{k+1}$-concave on } \Gamma_n.\]
\end{thm}
\begin{rem}
$\bar{f}_{\vec{a}} = a_k \binom{n}{k} X^k + a_{k+1} \binom{n}{k+1} X^{k+1}$ is always real-rooted, so this fulfils conjectures \ref{mainconj1} and \ref{mainconj2}.
\end{rem}
\begin{proof}
We write
\[a_k \sigma_k + a_{k+1} \sigma_{k+1} = a_k \sigma_k \left(1 + \dfrac{a_{k+1}}{a_k}\dfrac{\sigma_{k+1}}{\sigma_k}\right).\]
However $\sigma_k$ is $\frac{1}{k}$-concave on $\Gamma_n$ according to theorem \ref{concsigmak1surk}, and $\sigma_{k+1}/\sigma_k$ is $1$-concave according to theorem \ref{concsigmak+1sigmak}; so is $1 + a_{k+1}\sigma_{k+1}/a_k\sigma_k$.

Therefore, referring to lemma \ref{concgamma}, $a_k \sigma_k (1 + a_{k+1}\sigma_{k+1}/a_k\sigma_k)$ is $\frac{1}{k+1}$-concave.
\end{proof}

\subsubsection{Roots and \texorpdfstring{$\sigma_k$}{sk}}
Now we shall use the connection between the roots of a polynomial and the elementary symmetric polynomials. We shall need a formula:
\begin{lem}
Let $\lambda \in \Rb^l$, $\mu \in \Rb^m$, $k \in \Nb$. Then
\begin{equation} \label{sumsigma}
\sigma_k(\lambda, \mu) = \sigma_k(\lambda)\sigma_0(\mu) + \sigma_{k-1}(\lambda)\sigma_1(\mu) + \ldots + \sigma_1(\lambda)\sigma_{k-1}(\mu) + \sigma_0(\lambda)\sigma_k(\mu).
\end{equation}
\end{lem}
\begin{proof}
This can directly be proved, eg. by counting arguments or by double induction on $k$ and $l$.
\end{proof}

\begin{thm} \label{Pascinde}
Let $\vec{a} = (a_0, a_1, \ldots, a_p) \in \Rb^{p+1}$. If there exist $m \in \Nb$, $R \in \Rb_+[X] \cap \Rb_{m-1}[X]$ such that
\[Q := X^mP_{\vec{a}} + R\]
is real-rooted, then $f_{\vec{a}}^{1/p}$ is concave on $\Gamma_n$.

Equivalently: if there exist $m \in \Nb$, $\lambda \in \left(\Rb_+\right)^{m+p}$ such that
\[\dfrac{a_k}{a_p} = \sigma_{p-k}(\lambda), \qquad 0 \leq k \leq p,\]
then $a_0 + a_1\sigma_1 + \ldots + a_p\sigma_p$ is $\frac{1}{p}$-concave on $\Gamma_n$.
\end{thm}
\begin{rem}
This theorem implies theorem \ref{akak+1}.
\end{rem}
\begin{proof}
Actually we can show something stronger. Let $\lambda = (\lambda_1, \ldots, \lambda_{m+p}) \in \Rb_+^{m+p}$ such that $-\lambda$ is the set of the roots of $Q$. Then
\[Q = a_p(X+\lambda_1)\ldots(X+\lambda_{m+p}) = a_p\left[\sigma_0(\lambda)X^{m+p} + \sigma_1(\lambda)X^{m+p-1} + \ldots + \sigma_p(\lambda)X^m + \ldots + \sigma_{m+p}(\lambda)X^0\right].\]
However,
\[Q = X^mP + R = a_p X^{m+p} + a_{p-1} X^{m+p-1} + \ldots + a_0 X^m + R,\]
so
\[a_p \sigma_k(\lambda) = a_{p-k}, \qquad 0 \leq k \leq p.\]
Therefore,
\begin{align*}
f_{\vec{a}}(x) &= a_0 \sigma_0(x) + a_1 \sigma_1(x) + \ldots + a_p \sigma_p(x) \\
&= a_p\left[\sigma_p(\lambda)\sigma_0(x) + \sigma_{p-1}(\lambda)\sigma_1(x) + \ldots + \sigma_0(\lambda) \sigma_p(x)\right] \\
&= a_p \sigma_p(\lambda,x)
\end{align*}
according to \eqref{sumsigma}. We know by theorem \ref{concsigmak1surk} that $(\mu,x) \in \left(\Rb_+^*\right)^{m+p+n} \longmapsto \sigma_p(\mu,x)$ is a $\frac{1}{p}$-concave function on $\left(\Rb_+^*\right)^{m+p+n}$, let alone its restriction: $x \in \Gamma_n \longmapsto \sigma_p(\lambda,x)$, which is proportional to $f_{\vec{a}}$.
\end{proof}

Hence, given $\vec{a} \in \Rb_+^{p+1}$, we can use the criteria on the real-rootedness of polynomials from section \ref{sectionpolynomials} to study $P_{\vec{a}}$, and conclude about the concavity of $f_{\vec{a}}^{1/p}$. For instance, the log-concavity of the sequence $(a_0, a_1, \ldots, a_p)$ is a necessary condition for every polynomial of the form $X^mP_{\vec{a}} + R$ to be real-rooted.

\begin{rem} \label{remPm}
Nevertheless, the fact that some polynomial of the form $X^mP_{\vec{a}} + R$ is real-rooted is too strong to be a necessary condition to the concavity of $f_{\vec{a}}^{1/p}$. For example, let us take $n=2$, $\vec{a} = (1,p,q) \in \Rb_+^3$. Then
\begin{align*}
f_{\vec{a}}(x_1,x_2) &= 1 + p(x_1+x_2) + qx_1x_2, \\
H\left(f_{\vec{a}}^{1/2}\right)(x) &\equiv -\dfrac{1}{2}\begin{pmatrix} (p+qx_2)^2 & (p+qx_1)(p+qx_2) - 2qf_{\vec{a}}(x) \\ (p+qx_1)(p+qx_2) - 2qf_{\vec{a}}(x) & (p+qx_1)^2 \end{pmatrix}, \\
\det H\left(f_{\vec{a}}^{1/2}\right) &\equiv p^2 - q, \\
P_{\vec{a}} &= 1 + pX + qX^2, \\
(X^m P_{\vec{a}})^{(m)} &= m!\left(\binom{m}{0} + p\binom{m+1}{1}X + q\binom{m+2}{2}X^2\right).
\end{align*}
So the discriminant of $(X^mP_{\vec{a}})^{(m)}$ is proportional to $p^2 - 2q\left(1+\frac{1}{m+1}\right)$, which is smaller than $p^2 - 2q$.

Let us choose $p=2$ and $q=3$. Then $f_{\vec{a}}^{1/2}$ is concave. Yet $(X^m P_{\vec{a}})^{(m)}$ is not real-rooted, thus no polynomial of the form $X^mP_{\vec{a}} + R$ can be real-rooted, according to Rolle's lemma.
\end{rem}

\begin{rem}
Let $P \in \Rb[X]$, $m \in \Nb$. Rolle's lemma guarantees the following assertion:
\[\text{There exists some } R \in \Rb_{m-1}[X] \text{ such that } X^mP + R \text{ is real-rooted}. \qquad \Longrightarrow \qquad (X^mP)^{(m)} \text{ is real-rooted}.\]
But the reverse is not true. Let us take $Q = X(X-1)(X-2)(X-4)(X-5)(X-6)$. $Q$ is real-rooted, although none of its primitives is. Let us set $S = \int_0^X Q$. $S$ is a polynomial that vanishes in $0$, so there exists some $P \in \Rb[X]$ such that $S = XP$. Hence $(XP)' = Q$ is real-rooted, but there is no $r \in \Rb$ such that $XP+r = r + \int_0^X Q$ is real-rooted.
\end{rem}

The hypothesis of theorem \ref{Pascinde} is strong: it amounts to assume that $f_{\vec{a}}$ is the restriction of some $\sigma_p$. However, it is not the only way to be $\frac{1}{p}$-concave. In the next section, we shall develop a new conjecture based on proposition \ref{concroots}. This conjecture is weaker than conjecture \ref{mainconj1}.

\section{Homogeneous polynomials} \label{sectionhomog}

\setcounter{conjbis}{1}
\begin{conjbis}
Let $p \in \Nb^*$, $\vec{a} = (a_0, a_1, \ldots, a_p) \in (\Rb_+)^{p+1}$. Then
\[\bar{f}_{\vec{a}} \text{ is real-rooted} \quad \Longrightarrow \quad f_{\vec{a}} \text{ is $\frac{1}{p}$-concave on } \Gamma_n.\]
That is to say,
\[\Xi_n^p \subset \Kc_n^p.\]

Or equivalently: if there exists $\lambda \in \left(\Rb_+\right)^{p}$ such that
\[\dfrac{a_k}{a_p} = \dfrac{\binom{n}{p}}{\binom{n}{k}}\sigma_{p-k}(\lambda), \qquad 0 \leq k \leq p,\]
then $a_0 + a_1\sigma_1 + \ldots + a_p\sigma_p$ is $\frac{1}{p}$-concave on $\Gamma_n$.
\end{conjbis}

We could only prove this conjecture for $p=2$, in which case the theorem \ref{thmdeg123} makes it equivalent to conjecture \ref{mainconj1}. However, the tools are different and promising for a full resolution of the conjecture.

We need to introduce some elements of the theory of hyperbolic polynomials, opened by L. G\aa rding.

\subsection{Hyperbolic polynomials}
See the seminal paper of \cite{Gar59}, or self-contained summaries in \cite{Caf85} or \cite{Harvey09}.

\begin{defi}
Let $v \in \Rb^n$, $P \in \Rb[X_1, \ldots, X_n]$. $P$ is said to be $v$-hyperbolic if
\begin{enumerate}[(i)]
 \item $P$ is homogeneous,
 \item $P(v) > 0$,
 \item for every $x \in \Rb^n$, the polynomial $P(x + X\cdot v) \in \Rb[X]$ is real-rooted.
\end{enumerate}

We write $\Gamma_v := \left\{x \in \Rb^n \quad \big| \quad P(x) > 0\right\}$. $\Gamma_v$ is called the G\aa rding cone of $P$.
\end{defi}

\begin{prop}[1959, \cite{Gar59}] \label{hypconc}
Let $P$ be a $v$-hyperbolic polynomial. Then $\Gamma_v$ is a convex cone with vertex at 0, and $P$ is $w$-hyperbolic for all $w \in \Gamma_v$.

Moreover, $P$ is $\frac{1}{\deg P}$-concave on $\Gamma_v$.
\end{prop}
This is shown in \cite{Gar59}, and an explicit formulation can be found in \cite{Caf85} or \cite{Harvey09}.

One will now establish this property for the $\sigma_k$ functions.
\begin{lem}
Let $\lambda \in \Rb^l$, $k \in \Nb$. Then
\begin{align}
\sigma_k(\lambda + X\cdot\Ib) &= \dfrac{1}{(n-k)!} \sigma_n(\lambda + X\cdot\Ib)^{(n-k)} \label{sigmakXder} \\
&= \binom{l-k}{0}\sigma_k(\lambda)X^0 + \binom{l-k+1}{1}\sigma_{k-1}(\lambda)X^1 + \ldots + \binom{l-1}{k-1}\sigma_1(\lambda)X^{k-1} + \binom{l}{k}\sigma_0(\lambda)X^k. \label{sigmakXsum}
\end{align}
\end{lem}

\begin{proof}
\eqref{sigmakXder} comes from the fact that for $x \in \Rb^l$,
\[\drond_{x_i} \sigma_{k+1}(x) = \sigma_{k}(x'_i).\]
Yet, counting each term leads to
\[\sum_{i=1}^n \sigma_{k}(x'_i) = (n-k)\sigma_k(x),\]
so
\begin{align*}
\sigma_{k+1}(\lambda + X\cdot\Ib)'&= \sum_{i=1}^n \drond_{x_i}\sigma_{k+1}(\lambda + X\cdot\Ib) \\
&= (n-k)\sigma_{k}(\lambda + X\cdot\Ib).
\end{align*}

\eqref{sigmakXsum} can be proved by double induction on $k$ and $l$:
\begin{align*}
\sigma_k(\lambda + X\cdot\Ib) &= (\lambda_1 + X)\sigma_{k-1}(\lambda'_1 + X\cdot\Ib'_1) + \sigma_{k}(\lambda'_1 + X\cdot\Ib'_1) \\
&= (\lambda_1 + X)\sum_{i=0}^{k-1}\binom{(l-1)-(k-1)+i}{i}\sigma_{k-1-i}(\lambda'_1)X^i \\
&\qquad + \sum_{j=0}^k \binom{(l-1)-k+j}{j}\sigma_{k-j}(\lambda'_1)X^j \quad \text{by induction hypothesis,} \\
&= \sum_{i=0}^{k-1}\left[\binom{l-k+i}{i}\sigma_{k-1-i}(\lambda'_1)\lambda_1 + \binom{l-1-k+i}{i}\sigma_{k-i}(\lambda'_1)\right]X^i \\
&\qquad + \sum_{i=0}^{k-1}\binom{l-k+i}{i}\sigma_{k-1-i}(\lambda'_1)X^{i+1} + \binom{l-1-k+k}{k}\sigma_{k-k}(\lambda'_1)X^k \\
&= \left[\binom{l-k+0}{0}\sigma_{k-1-0}(\lambda'_1)\lambda_1 + \binom{l-1-k+0}{0}\sigma_{k-0}(\lambda'_1)\right]X^0 \\
&\qquad + \sum_{i=1}^{k-1}\left[\binom{l-k+i}{i}\sigma_{k-1-i}(\lambda'_1)\lambda_1 + \binom{l-1-k+i}{i}\sigma_{k-i}(\lambda'_1) + \binom{l-k+i-1}{i-1}\sigma_{k-i}(\lambda'_1)\right]X^i \\
&\qquad + \binom{l-k+k-1}{k-1}\sigma_{k-k}(\lambda'_1)X^k + \binom{l-1-k+k}{k}\sigma_{k-k}(\lambda'_1)X^k \\
&= \sigma_k(\lambda) X^0 + \sum_{i=1}^{k-1}\binom{l-k+i}{i}\sigma_{k-i}(\lambda)X^i + \binom{l}{k}X^k \\
&= \sum_{i=0}^{k}\binom{l-k+i}{i}\sigma_{k-i}(\lambda)X^i.
\end{align*}
It can also be derived directly from \eqref{sigmakXder}: let $0 \leq p \leq n$.
\begin{align*}
\sigma_n(\lambda + X\cdot\Ib)^{(p)} &= \left[(X+\lambda_1)\ldots(X+\lambda_n)\right]^{(p)} \\
&= \left[\sigma_n(\lambda)X^0 + \sigma_{n-1}(\lambda)X^1 + \ldots + \sigma_{0}(\lambda)X^n\right]^{(p)} \\
&= 0 + 0 + \ldots + p(p-1)\ldots.2.1\sigma_{n-p}(\lambda)X^0 + \ldots + n(n-1)\ldots(n-p+1)\sigma_{0}(\lambda)X^{n-p} \\
&= p!\left[\binom{p}{0} \sigma_{n-p}(\lambda)X^0 + \ldots + \binom{n}{n-p}\sigma_{0}(\lambda)X^{n-p}\right].
\end{align*}
\end{proof}

Hence one can deduce this theorem:
\begin{thm}[1959, \cite{Gar59}] \label{sigmakhyp}
For $0 \leq k \leq n$, $\sigma_k$ is $\Ib$-hyperbolic.
\end{thm}
\begin{proof}
It is only \eqref{sigmakXder} with Rolle's lemma.
\end{proof}

\begin{rem}
One can hold the same reasoning for a general hyperbolic polynomial: for $v=(v_1, \ldots, v_n) \in \Rb^n$ and $P$ a $v$-hyperbolic polynomial, $Q = \sum_i v_i \drond_{x_i}P$ is $v$-hyperbolic as well.
\end{rem}

\begin{cor} \label{sigmaxsigmamu}
For all $x \in \Rb^n$, $0 \leq k \leq n$, there exists $\mu \in \Rb^k$ such that for all $0 \leq l \leq k$,
\[\sigma_l(x) = \dfrac{\binom{n}{k}}{\binom{n-l}{k-l}}\sigma_l(\mu).\]
\end{cor}
\begin{proof}
Let $x \in \Rb^n$, $0 \leq k \leq n$. According to theorem \ref{sigmakhyp}, $\sigma_k(x + X\cdot\Ib)$ is real-rooted, so there exist $\alpha \in \Rb$, $\mu \in \Rb^n$ such that
\begin{align*}
\sigma_k(x + X\cdot\Ib) &= \alpha(X+\mu_1)\ldots(X+\mu_k) \\
&= \alpha\left[\sigma_k(\mu)X^0 + \ldots + \sigma_0(\mu)X^k\right].
\end{align*}
Meanwhile, \eqref{sigmakXsum} gives that
\[\sigma_k(x + X\cdot\Ib) = \binom{n-k}{k-k}\sigma_k(x)X^0 + \binom{n-k+1}{k-k+1}\sigma_{k-1}(x)X^1 + \ldots + \binom{n-1}{k-1}\sigma_1(x)X^{k-1} + \binom{n}{k}\sigma_0(x)X^k.\]
Identifying the coefficients, we obtain that $\alpha = \binom{n}{k}$, and for all $0 \leq l \leq k$,
\[\binom{n-l}{k-l}\sigma_l(x) = \binom{n}{k} \sigma_l(\mu).\]
\end{proof}

Let us come back to our main problem. One way to show that $f_{\vec{a}}$ is $\frac{1}{p}$-concave would be to prove that it is a hyperbolic polynomial. However, $f_{\vec{a}}$ does not satisfy the first essential property of hyperbolic polynomial: it is not homogeneous. In theorem $\ref{Pascinde}$, we introduced new variables that made $f_{\vec{a}}$ homogeneous by assuming that $P_{\vec{a}}$ is real-rooted. In the next paragraph we shall suppose something weaker: that $\bar{f}_{\vec{a}}$ is real-rooted.

\subsection{Semi-symmetric polynomials}
\begin{lem}
Let $\vec{a} = (a_0, a_1, \ldots, a_p) \in \Rb_+^{p+1}$. Let us suppose that
\[\bar{f}_{\vec{a}} = \binom{n}{0}a_0 + \binom{n}{1}a_1X + \ldots + \binom{n}{p}a_pX^p\]
is real-rooted. Then there exists some $\lambda = (\lambda_1, \ldots, \lambda_p) \in \Rb_+^p$ such that
\[\frac{\binom{n}{k}a_k}{\binom{n}{p}a_p} = \sigma_{p-k}(\lambda), \qquad 0 \leq k \leq p.\]
Thus
\begin{align*}
f_{\vec{a}}(x) &= a_0 + a_1 \sigma_1(x) + \ldots + a_p \sigma_p(x) \\
&= \binom{n}{p}a_p\left[\frac{\sigma_0(x)\sigma_{p-0}(\lambda)}{\binom{n}{0}} + \frac{\sigma_1(x)\sigma_{p-1}(\lambda)}{\binom{n}{1}} + \ldots + \frac{\sigma_p(x)\sigma_{p-p}(\lambda)}{\binom{n}{p}}\right].
\end{align*}
\end{lem}

We introduce the following functions.
\begin{defi}
Let $n, p \in \Nb$. We call \textbf{semi-symmetric polynomials} the polynomials
\[s_{n,p} : (x, \lambda) \in \Rb^n \times \Rb^p \longmapsto \sum_{k=0}^p \dfrac{\sigma_k(x)\sigma_{p-k}(\lambda)}{\binom{n}{k}}.\]
$s_{n,p} \in \Rb[X_1, \ldots, X_n,X_{n+1}, \ldots, X_{n+p}]$ is $p$-homogeneous, symmetric in the coefficients of $x$ and $\lambda$, but not between each other.

For $\mu, \lambda \in \Rb^p$, we define
\[\pi_p(\mu, \lambda) = s_{p,p}(\mu + X\cdot\Ib,\lambda + X\cdot\Ib) \in \Rb[X].\]
\end{defi}

Then the conjecture \ref{mainconj2} is implied by the much wider statement:
\begin{conj} \label{mainconj4}
$s_{n,p}$ is a $\Ib$-hyperbolic polynomial.
\end{conj}
With such a result $s_{n,p}$ would be $\frac{1}{p}$-concave, and its restriction to $\Rb^n$ as well, ie. $f_{\vec{a}}$.

\begin{thm} \label{pip}
Conjecture \ref{mainconj4} is equivalent to show that $s_{p,p}$ is a $\Ib$-hyperbolic polynomial, or equivalently that $\pi_p(\mu, \lambda)$ is real-rooted, for all $\lambda, \mu \in \Rb^p$.

Moreover,
\[\pi_p(\mu,\lambda) = \sum_{l=0}^p \left(\sum_{i=0}^{p-l} \dfrac{(p-i)!(l+i)!}{p!l!} \sigma_{p-l-i}(\mu)\sigma_i(\lambda)\right) (2X)^l.\]
\end{thm}
\begin{proof}
Let $x \in \Rb^n$, $\lambda \in \Rb^p$. Then
\begin{align*}
s_{n,p}(x+X\cdot\Ib, \lambda+X\cdot\Ib) &= \sum_{k=0}^p \dfrac{\sigma_k(x+X\cdot\Ib)\sigma_{p-k}(\lambda+X\cdot\Ib)}{\binom{n}{k}} \\
&= \sum_{k=0}^p \binom{n}{k}^{-1} \left(\sum_{i=0}^k\binom{n-k+i}{i}\sigma_{k-i}(x)X^i\right)\left(\sum_{j=0}^{p-k}\binom{p-(p-k)+j}{j}\sigma_{(p-k)-j}(\lambda)X^j\right) \\
&= \sum_{k=0}^p \sum_{i=0}^k \sum_{j=0}^{p-k} \dfrac{(n-k+i)!(k+j)!}{n!i!j!}\sigma_{k-i}(x)\sigma_{(p-k)-j}(\lambda) X^{i+j}.
\end{align*}
We used \eqref{sigmakXsum}. We proceed to the change of variable $l = i+j$ and use the fact that $\sigma_{k-i}(x) = 0$ when $i > k$:
\begin{align*}
s_{n,p}(x+X\cdot\Ib, \lambda+X\cdot\Ib)
&= \sum_{k=0}^p \sum_{l=0}^p \sum_{i=0}^{\min(k,l)} \dfrac{(n-k+i)!(k+l-i)!}{n!i!(l-i)!}\sigma_{k-i}(x)\sigma_{p-k-l+i}(\lambda) X^l \\
&= \sum_{k=0}^p \sum_{l=0}^p \sum_{i=0}^l \dfrac{(n-k+i)!(k+l-i)!}{n!i!(l-i)!}\sigma_{k-i}(x)\sigma_{p-k-l+i}(\lambda) X^l.
\end{align*}
Setting $m = p-k-l+i$, we get
\begin{align*}
s_{n,p}(x+X\cdot\Ib, \lambda+X\cdot\Ib)
&= \sum_{l=0}^p \sum_{i=0}^l \sum_{m=-l+i}^{p-l+i} \dfrac{(n+m-p+l)!(p-m)!}{n!i!(l-i)!}\sigma_{p-l-m}(x)\sigma_{m}(\lambda) X^l \\
&= \sum_{l=0}^p \sum_{i=0}^l \sum_{m=0}^{p-l} \dfrac{(n+m-p+l)!(p-m)!}{n!i!(l-i)!}\sigma_{p-l-m}(x)\sigma_{m}(\lambda) X^l \\
&= \sum_{l=0}^p \sum_{m=0}^{p-l} \left(\sum_{i=0}^l \dfrac{1}{i!(l-i)!} \right)\dfrac{(n+m-p+l)!(p-m)!}{n!}\sigma_{p-l-m}(x)\sigma_{m}(\lambda) X^l \\
&= \sum_{l=0}^p \sum_{m=0}^{p-l} \left(\dfrac{2^l}{l!}\right) \dfrac{(n+m-p+l)!(p-m)!}{n!}\sigma_{p-l-m}(x)\sigma_{m}(\lambda) X^l \\
&= \sum_{l=0}^p \sum_{m=0}^{p-l} \dfrac{(n+m-p+l)!(p-m)!}{n!l!}\sigma_{p-l-m}(x)\sigma_{m}(\lambda) (2X)^l.
\end{align*}
We apply the corollary \ref{sigmaxsigmamu} and find that there exists $\mu \in \Rb^p$ such that for all $0 \leq i \leq p$,
\[\sigma_i(x) = \dfrac{\binom{n}{p}}{\binom{n-i}{p-i}}\sigma_i(\mu).\]
Hence
\begin{align*}
s_{n,p}(x+X\cdot\Ib, \lambda+X\cdot\Ib)
&= \sum_{l=0}^p \sum_{m=0}^{p-l} \dfrac{(n+m-p+l)!(p-m)!}{n!l!} \dfrac{\binom{n}{p}}{\binom{n-(p-l-m)}{p-(p-l-m)}}\sigma_{p-l-m}(\mu)\sigma_{m}(\lambda) (2X)^l \\
&= \sum_{l=0}^p \sum_{m=0}^{p-l} \dfrac{(p-m)!(m+l)!}{l!p!} \sigma_{p-l-m}(\mu)\sigma_{m}(\lambda) (2X)^l \\
&= s_{p,p}(\mu+X\cdot\Ib,\lambda+X\cdot\Ib).
\end{align*}

If $s_{p,p}$ is a hyperbolic polynomial, then $s_{p,p}(\mu+X\cdot\Ib,\lambda+X\cdot\Ib) = \pi_p(\mu,\lambda) = s_{n,p}(x+X\cdot\Ib, \lambda+X\cdot\Ib)$ is real-rooted, and $s_{n,p}$ is hyperbolic as well.
\end{proof}

\begin{rem}
Citing corollary \ref{sigmaxsigmamu}, we could have replaced $x \in \Rb^n$ by $\mu \in \Rb^p$ since the definition of $s_{n,p}$: given $x \in \Rb^n$, there is some $\mu \in \Rb^p$ such that for all $0 \leq k \leq p$, $\sigma_k(x) = \frac{\binom{n}{p}}{\binom{n-k}{p-k}}\sigma_k(\mu)$. Then
\[s_{n,p}(x,\lambda) = \sum_{k=0}^p \dfrac{\sigma_k(x)\sigma_{p-k}(\lambda)}{\binom{n}{k}} = \sum_{k=0}^p \dfrac{\sigma_{k}(\mu)\sigma_{p-k}(\lambda)}{\binom{p}{k}}.\]
However that would not have given the equivalence between the hyperbolicity of $s_{n,p}$ and $s_{p,p}$.
\end{rem}

\subsection{A new conjecture}

We found an other equivalent formulation of the conjecture \ref{mainconj4}.

\begin{thm} \label{thmPQscinde}
Conjecture \ref{mainconj4} is equivalent to the following assertion.

For all $P, Q \in \Rb[X]$ with $\deg P = \deg Q = p$, if $P$ and $Q$ are real-rooted, then
\[\sum_{k=0}^p P^{(k)} Q^{(p-k)}\]
is real-rooted as well.
\end{thm}
\begin{proof}
We know by theorem \ref{pip} that conjecture \ref{mainconj4} is equivalent to the real-rootedness of $\pi_p(\mu,\lambda)$ for all $\mu, \lambda \in \Rb^p$. Yet, by \eqref{sigmakXder},
\begin{align*}
\pi_p(\mu,\lambda) &= s_{p,p}(\mu + X\cdot\Ib, \lambda + X\cdot\Ib) \\
&= \sum_{k=0}^p \dfrac{\sigma_k(\mu + X\cdot\Ib) \sigma_{p-k}(\lambda + X\cdot\Ib)}{\binom{p}{k}} \\
&= \sum_{k=0}^p \dfrac{\sigma_p(\mu + X\cdot\Ib)^{(p-k)} \sigma_p(\lambda + X\cdot\Ib)^{(k)}}{(p-k)!\binom{p}{k}k!} \\
&= \dfrac{1}{p!} \sum_{k=0}^p \left[(X+\mu_1)\ldots(X+\mu_p)\right]^{(p-k)} \left[(X+\lambda_1)\ldots(X+\lambda_p)\right]^{(k)},
\end{align*}
and $(X+\mu_1)\ldots(X+\mu_p)$ and $(X+\lambda_1)\ldots(X+\lambda_p)$ are the most general form of real-rooted polynomials with degree $p$.
\end{proof}

This assertion looks like Rolle's lemma, and might have a generalisation \textit{\`a la} Gauss-Lucas, such as
\begin{conj} \label{conjgausslucas}
For all $P, Q \in \Cb[X]$ with $\deg P = \deg Q = p$, the roots of $\sum_{k=0}^p P^{(k)} Q^{(p-k)}$ lie in the convex hull of those of $P$ and $Q$.
\end{conj}

\begin{prop} \label{conj4p2}
Conjecture \ref{mainconj4} is true for $p=2$.
\end{prop}
\begin{proof}
Let $\mu_1, \mu_2, \lambda_1, \lambda_2 \in \Rb$, $P = (X+\lambda_1)(X+\lambda_2)$, $Q = (X+\mu_1)(X+\mu_2)$. Then
\begin{align*}
P' &= 2X + \lambda_1 + \lambda_2, \\
Q' &= 2X + \mu_1 + \mu_2, \\
P'' = Q'' &= 2, \\
PQ'' + P'Q' + P''Q &= 8X^2 + 4(\lambda_1 + \lambda_2 + \mu_1 + \mu_2)X + 2\lambda_1\lambda_2 + 2\mu_1\mu_2 + (\lambda_1+\lambda_2)(\mu_1+\mu_2).
\end{align*}
The discriminant of $PQ'' + P'Q' + P''Q$ is
\[\Delta_{PQ'' + P'Q' + P''Q} = 16\left[(\lambda_1 - \lambda_2)^2 + (\mu_1 - \mu_2)^2\right] \geq 0,\]
so $PQ'' + P'Q' + P''Q$ is real-rooted. Moreover, its roots are
\[x_{\pm} = -\dfrac{\lambda_1 + \lambda_2 + \mu_1 + \mu_2}{4} \pm \dfrac{\sqrt{(\lambda_1 - \lambda_2)^2 + (\mu_1 - \mu_2)^2}}{4}.\]
\end{proof}

\begin{rem}
This is an other way to show conjecture \ref{mainconj2} for $p=2$.
\end{rem}

Roughly speaking, the corollary \ref{corlogconc} means that ``the more'' a polynomial is log-concave, the more it is likely to be real-rooted. The main example of a log-concave sequence is the sequence of the binomial coefficients: $\displaystyle{\binom{n}{k}}$ is log-concave, and
\[\binom{n}{0} + \binom{n}{1}X + \ldots + \binom{n}{n-1}X^{n-1} + \binom{n}{n}X^n = (1+X)^n\]
is real-rooted. Moreover, multiplying by the binomial sequence preserves the real-rootedness: it makes a polynomial ``more'' log-concave. This is a classical result of the theory of real-rooted polynomials, see for instance \cite{Fisk08} in chapter 7.
\begin{thm}[\cite{Fisk08}]
Let $P = a_0 + a_1 X + \ldots + a_p X^p \in \Rb_+[X]$, $n \in \Nb^*$. If $P$ is real-rooted, then
\[\binom{n}{0}a_0 + \binom{n}{p}a_1 X + \ldots + \binom{n}{p}a_p X^p \qquad \text{is real-rooted as well}.\]
\end{thm}
Here we do the contrary: we divide by the binomial sequence, hence we make the polynomial ``less'' log-concave, and less likely to be real-rooted. We have already seen that
\[\sum_{k=0}^p \sigma_k(\mu) \sigma_{p-k}(\lambda) = \sigma_p(\mu,\lambda)\]
is hyperbolic, and therefore $\displaystyle{\sum_{k=0}^p \sigma_k(\mu + X \cdot \Ib) \sigma_{p-k}(\lambda + X \cdot \Ib)}$ is real-rooted. In terms of polynomials $P$ and $Q$ from theorem \ref{thmPQscinde}, it writes simply
\[\sum_{k=0}^p \binom{p}{k}Q^{(p-k)}P^{(k)} = (QP)^{(p)},\]
which is real-rooted by $p$ applications of Rolle's lemma. If our conjecture is true and
\[\sum_{k=0}^p \dfrac{\sigma_k(\mu + X \cdot \Ib) \sigma_{p-k}(\lambda + X \cdot\Ib)}{\binom{p}{k}} = \sum_{k=0}^p Q^{(p-k)}P^{(k)}\]
is real-rooted, this could imply that the semi-symmetric polynomials are real-rooted although not much log-concave. The proposition \ref{conj4p2} even suggests that they could be optimal, in some sense, among the real-rooted polynomials.

\section{Summary and conclusion}
In this paper was raised the question of the $\frac{1}{p}$-concavity of the functions
\[f_{\vec{a}} : x \in \Gamma_n \longmapsto \sum_{k=0}^p a_k \sigma_k(x),\]
with $\vec{a} = (a_0, a_1, \ldots, a_p) \in (\Rb_+)^{p+1}$. In section \ref{sectionintro} we formulated the following conjectures:
\setcounter{conjbis}{0}
\begin{conjbis}
\begin{align*}
f_{\vec{a}} \text{ is $\frac{1}{p}$-concave on } \Gamma_n \quad &\Longleftrightarrow \quad f_{\vec{a}} \text{ is $\frac{1}{p}$-concave on } \Delta \\
&\Longleftrightarrow \quad \bar{f}_{\vec{a}} \text{ is $\frac{1}{p}$-concave on } \Rb_+^*.
\end{align*}
\end{conjbis}
\begin{conjbis}
\[f_{\vec{a}} \text{ is $\frac{1}{p}$-concave on } \Gamma_n \quad \Longleftarrow \quad \bar{f}_{\vec{a}} \text{ is real-rooted}.\]
\end{conjbis}
We showed with proposition \ref{concroots} that for all $p, n$,
\[\text{Conjecture } \ref{mainconj2} \quad \Longrightarrow \quad \text{Conjecture } \ref{mainconj1}.\]

Later in section \ref{sectionhomog} we introduced the homogeneous semi-symmetric polynomials
\[s_{n,p} : (x, \lambda) \in \Rb^n \times \Rb^p \longmapsto \sum_{k=0}^p \dfrac{\sigma_k(x) \sigma_{p-k}(\lambda)}{\binom{n}{k}},\]
and emitted the following:
\stepcounter{conjbis}
\begin{conjbis}
$s_{n,p}$ is a $\Ib$-hyperbolic polynomial.
\end{conjbis}
We showed in theorem \ref{pip} that this is equivalent to
\addtocounter{conjbis}{-1}
\begin{conjbis}
$s_{p,p}$ is a $\Ib$-hyperbolic polynomial.
\end{conjbis}
Or, equivalently, according to theorem \ref{thmPQscinde}, to
\addtocounter{conjbis}{-1}
\begin{conjbis}
For all $P, Q \in \Rb[X]$ with $\deg P = \deg Q = p$, if $P$ and $Q$ are real-rooted, then
\[\sum_{k=0}^p P^{(k)} Q^{(p-k)}\]
is real-rooted as well.
\end{conjbis}
A natural generalisation of this result would be
\begin{conjbis}
For all $P, Q \in \Cb[X]$ with $\deg P = \deg Q = p$, the roots of $\sum_{k=0}^p P^{(k)} Q^{(p-k)}$ lie in the convex hull of those of $P$ and $Q$.
\end{conjbis}
The theory of hyperbolic polynomials ensures that
\[\text{Conjecture } \ref{mainconj4} \quad \Longrightarrow \quad \text{Conjecture } \ref{mainconj2},\]
and it is direct to see that 
\[\text{Conjecture } \ref{conjgausslucas} \quad \Longrightarrow \quad \text{Conjecture } \ref{mainconj4}.\]

Thus, the nature of $\Kc_n^p$ is connected to algebraic results on real-rootedness of polynomials. The section \ref{sectionpolynomials} recalls some known results and two original ones on necessary or sufficient criteria of real-rootedness. This made us formulate an other conjecture, unrelated to the previous ones.
\setcounter{conjbis}{2}
\begin{conjbis}
Let $P \in \Rb_+[X], \deg P = n$. Then
\[P \text{ is real-rooted} \qquad \Longleftarrow \qquad PP'' + \left(\dfrac{1}{n}-1\right)P'^2 \leq 0 \text{ on } \Rb.\]
\end{conjbis}
We established in theorem \ref{thmdeg123} that
\[\text{Conjecture \ref{conj3} is true for } \deg P = 2 \text{ or } 3.\]

In section \ref{sectionsigma}, we showed by the computation of the determinant of the hessian that
\[\text{Conjectures \ref{mainconj1} and \ref{mainconj2} are true for } p = 2.\]
That is to say,
\begin{align*}
\Kc_n^2 &= \left\{\vec{a} = (a_0,a_1,a_2) \in (\Rb_+)^3 \quad \big| \quad \bar{f}_{\vec{a}} \text{ is real-rooted} \right\} \\
&= \left\{\vec{a} = (a_0,a_1,a_2) \in (\Rb_+)^3 \quad \big| \quad na_1^2 - 2(n-1)a_0a_2 \geq 0 \right\}.
\end{align*}
For $p \geq 3$, we did not manage to get such a complete description of $\Kc_n^p$, but we determined some of its subsets that are quite wide. This is theorem \ref{Pascinde}, according to which
\[\left\{\vec{a} \in (\Rb_+)^{p+1} \quad \big| \quad (a_0, a_1, \ldots, a_p) \begin{array}{l}\text{ is the first part of the sequence of some} \\ \text{ polynomial with all roots real and non-positive} \end{array} \right\} \subset \Kc_n^p.\]
Some sufficient criterion for a polynomial to be real-rooted then implies the belonging to $\Kc_n^p$. For instance, Kurtz criterion gives
\[\left\{\vec{a} \in (\Rb_+)^{p+1} \quad \big| \quad \forall~ 1 \leq k \leq n-1, \quad 4a_{k-1}a_{k+1} < a_k^2 \right\} \subset \Kc_n^p.\]

In section \ref{sectionhomog} we proved the equivalence between the different formulations of conjecture \ref{mainconj4}, which is the widest algebraic description that we think is holding for $\Kc_n^p$. We established that
\[\text{Conjecture \ref{mainconj4} is true for } p = 2,\]
which is an other way to prove conjecture \ref{mainconj2} for $p=2$.

The five conjectures exposed in this paper bring to light unknown properties of the elementary symmetric polynomials, and promise interesting further research.

\section*{Appendix}
An usable formula for the general determinant of the hessian of $f_{\vec{a}}$ is still open and would be fruitful for the following of section \ref{sectionsigma}. In this appendix, we present two determinants that we manage to compute.

\subsection*{\texorpdfstring{$p=3$}{p=3}}

We have followed the same reasoning as in theorem \ref{thmp2} and computed the determinant for $p=3$, but the complexity of the determinant did not allow us to conclude about the concavity of $f_{\vec{a}}^{1/3}$.
\begin{prop}
Let $\vec{a} = (a_0, a_1, a_2, a_3) \in (\Rb_+)^4$. We set $f := f_{\vec{a}} = a_0 + a_1 \sigma_1 + a_2 \sigma_2 + a_3 \sigma_3 \in \Rb_+[X_1,\ldots,X_n]$.
%and \[\bar{f}_{\vec{a}}(X) := f_{\vec{a}}(X,\ldots,X) = a_0 \binom{n}{0} + a_1 \binom{n}{1} X + a_2 \binom{n}{2} X^2 + a_3 \binom{n}{3} X^3 \in \Rb_+[X].\]
Then
\begin{align*}
\det H(f^{1/3})(x) \equiv -\dfrac{1}{6}(-2f(x))^{n-1}\Bigg[&\dfrac{\sigma_n(y)}{\sigma_1(y)}\left(\sigma_1(y)\sum_i\frac{f_i(x)}{y_i} - (n-2)\sigma_1(D(f)(x))\right)^2 \\
&+ \left((n-2)^2\sigma_n(y) - \sigma_1(y)\sigma_{n-1}(y)\right)\left(\sum_i \frac{f_i(x)^2}{y_i} - \frac{\sigma_1(D(f)(x))}{\sigma_1(y)} + 3f(x)\right)\Bigg],
\end{align*}
where
\[y = \frac{a_2 + a_3\sigma_1(x)}{2}\Ib - a_3 x,\]
\end{prop}
\begin{proof}
We use the same notations as in lemma \ref{detHfalpha}. We have
\[H_{ii}(x) = 0, \qquad H_{i \neq j}(x) = a_2 + a_3(\sigma_1(x) - x_i - x_j).\]
We set $y_i = \frac{a_2 + a_3\sigma_1(x)}{2} - a_3 x_i$, so that $H_{i \neq j}(x) = y_i + y_j$, ie.
\[H(x) = \begin{pmatrix}
0 & y_1 + y_2 & \ldots & y_1 + y_n \\
y_2 + y_1 & 0 & & \vdots \\
\vdots & & \ddots & y_{n-1} + y_n \\
y_n + y_1 & \ldots & y_n + y_{n-1} & 0
\end{pmatrix} = y\otimes \Ib + \Ib \otimes y - 2 \Diag(y_1, \ldots, y_n).\]
Hence
\begin{align*}
\det H(x) = &\det[y_1\Ib + y - 2y_1e_1 \ | \ \ldots \ | \ y_n\Ib + y - 2y_ne_n] \\
	= &\det[- 2y_1e_1 \ | \ \ldots \ | \ - 2y_ne_n] \\
	&+ \sum_j \det[- 2y_1e_1 \ | \ \ldots \ | \ y \ | \ \ldots \ | \ - 2y_ne_n] \\
	&+ \sum_j \det[- 2y_1e_1 \ | \ \ldots \ | \ y_j \Ib \ | \ \ldots \ | \ - 2y_ne_n] \\
	&+ \sum_{i \neq j} \det[- 2y_1e_1 \ | \ \ldots \ | \ y_i \Ib \ | \ \ldots \ | \ y \ | \ \ldots \ | \ - 2y_ne_n] \\
	= &(-2)^n y_1\ldots y_n + 2\sum_j (-2)^{n-1}y_1 \ldots y_n \\
	&+ \sum_{i \neq j}(-2)^{n-2}y_1 \ldots y_{i-1}(y_iy_j - y_i^2)y_{i+1}\ldots y_{j-1}y_{j+1}\ldots y_n \\
	= &y_1\ldots y_n \left[(-2)^n + 2n(-2)^{n-1} + n(n-1)(-2)^{n-2} - (-2)^{n-2} \sum_{i \neq j} \frac{y_i}{y_j}\right] \\
	= &(-2)^{n-2}\left[(n-2)^2\sigma_n(y) - \sigma_1(y)\sigma_{n-1}(y)\right].
\end{align*}

Now we have to compute the comatrix of $H(x)$. Let us begin with $i=1$, $j=n$. We get
\begin{align*}
\bar{H}_{1n}(x) = &\det\begin{pmatrix}
y_2 + y_1 & 0 & y_2 + y_3 & \ldots & y_2 + y_{n-1} \\
y_3 + y_1 & y_3 + y_2 & 0 & \ldots & \vdots \\
& & y_4 + y_3 & \ddots & y_{n-2} + y_{n-1} \\
\vdots & \vdots & & & 0 \\
y_n + y_1 & y_n + y_2 & \ldots & y_n + y_{n-2} & y_n + y_{n-1}
\end{pmatrix} \\
= &\det[y'_1 + y_1\Ib \ | \ y'_1 + y_2\Ib -2y_2e_1 \ | \ \ldots \ | \ y'_1 + y_{n-1}\Ib - 2y_{n-1}e_{n-2}] \\
= &\det[y'_1 \ | \ -2y_2e_1 \ | \ \ldots \ | \ -2y_{n-1}e_{n-2}] \\
& +\sum_{j=2}^{n-1} \det[y'_1 \ | \ -2y_2e_1 \ | \ \ldots \ | \ y_j\Ib \ | \ \ldots \ | \ -2y_{n-1}e_{n-2}] \\
& +\sum_{j=2}^{n-1} \det[y_1\Ib \ | \ -2y_2e_1 \ | \ \ldots \ | \ y'_1 \ | \ \ldots \ | \ -2y_{n-1}e_{n-2}] \\
& +\det[y_1\Ib \ | \ -2y_2e_1 \ | \ \ldots \ | \ -2y_{n-1}e_{n-2}] \\
= &(-1)^ny_n(-2y_1)\ldots(-2y_{n-1}) \\
& +\sum_{j=2}^{n-1}(-2y_2)(-1)^{1+2}(-2y_3)(-1)^{1+2}\ldots (y_j^2 - y_ny_j)\ldots (-2y_{n-1})(-1)^{1+2} \\
& +(-1)\sum_{j=2}^{n-1}(-2y_2)(-1)^{1+2}(-2y_3)(-1)^{1+2}\ldots (y_jy_1 - y_ny_1)\ldots (-2y_{n-1})(-1)^{1+2} \\
& +(-1)^ny_1(-2y_2)\ldots(-2y_{n-1}) \\
= &2^{n-2}y_2\ldots y_n + 2^{n-2}y_1\ldots y_{n-1} \\
& + 2^{n-3}\sum_{j=2}^{n-1}y_2\ldots y_{n-1}(y_j - y_n)\left(1-\frac{y_1}{y_j}\right) \\
= &2^{n-3}y_2\ldots y_{n-1}\left[2y_n + 2y_1 + \sum_{j=2}^{n-1}\left(y_j - y_n - y_1 + \frac{y_1y_n}{y_j}\right)\right] \\
= &2^{n-3}\left[\frac{\sigma_n(y)\sigma_1(y)}{y_1y_n} - (n-2)\sigma_n(y)\left(\frac{1}{y_1} + \frac{1}{y_n}\right) + \sigma_{n-1}(y)\right].
\end{align*}
Similar calculus hold for any $i \neq j$:
\[\bar{H}_{ij}(x) = (-1)^{i+j}(-2)^{n-3}\left[\frac{\sigma_n(y)\sigma_1(y)}{y_iy_j} - (n-2)\sigma_n(y)\left(\frac{1}{y_i} + \frac{1}{y_j}\right) + \sigma_{n-1}(y)\right],\]
while for $i=j$ we recover the same form as the determinant of $H$:
\[\bar{H}_{ii}(x) = (-2)^{n-3}\left[(n-3)^2\sigma_{n-1}(y'_i) - \sigma_1(y'_i)\sigma_{n-2}(y'_i)\right],\]

We have seen in \eqref{Hgalpha} that
\[\det H(f^\alpha) \equiv f \det H + (\alpha-1) \sum_{i=1}^n \sum_{j=1}^n f_i f_j (-1)^{i+j} \bar{H}_{ij}.\]
So we compute
\begin{align*}
\sum_{i=1}^n &\sum_{j=1}^n f_i(x) f_j(x) (-1)^{i+j} \bar{H}_{ij}(x) \\
&= \sum_{i=1}^n f_i(x) \Bigg[\sum_{j \neq i} (-1)^{i+j} \left(f_j(x) (-1)^{i+j}(-2)^{n-3}\left[\frac{\sigma_n(y)\sigma_1(y)}{y_iy_j} - (n-2)\sigma_n(y)\left(\frac{1}{y_i} + \frac{1}{y_j}\right) + \sigma_{n-1}(y)\right]\right) \\
& \qquad + (-1)^{i+i}f_i(x) (-1)^{i+i}(-2)^{n-3}\left[(n-3)^2\sigma_{n-1}(y'_i) - \sigma_1(y'_i)\sigma_{n-2}(y'_i)\right]\Bigg] \\
&= \sum_{i=1}^n (-2)^{n-3} f_i(x) \Bigg[\left(\sum_{j \neq i} \frac{f_j(x)}{y_j}\right)\left(\frac{\sigma_n(y)\sigma_1(y)}{y_i} - (n-2)\sigma_n(y)\right) \\
& \qquad + \left(\sum_{j \neq i} f_j(x)\right)\left(\sigma_{n-1}(y) - (n-2)\frac{\sigma_n(y)}{y_i}\right) + f_i(x) \left[(n-3)^2\sigma_{n-1}(y'_i) - \sigma_1(y'_i)\sigma_{n-2}(y'_i)\right]\Bigg] \\
&= (-2)^{n-3} \sum_{i=1}^n f_i(x) \Bigg[\left(\sum_{j} \frac{f_j(x)}{y_j} - \frac{f_i(x)}{y_i}\right)\left(\frac{\sigma_n(y)\sigma_1(y)}{y_i} - (n-2)\sigma_n(y)\right) \\
& \qquad + \left(\sigma_1(D(f)(x)) - f_i(x)\right)\left(\sigma_{n-1}(y) - (n-2)\frac{\sigma_n(y)}{y_i}\right) \\
& \qquad + f_i(x) \left[(n-3)^2\sigma_{n-1}(y'_i) - (\sigma_1(y) - y_i)\frac{\sigma_{n-1}(y) - \sigma_{n-1}(y'_i)}{y_i}\right]\Bigg] \\
&= (-2)^{n-3} \Bigg[ \sigma_n(y)\sigma_1(y)\left[\left(\sum_i \frac{f_i(x)}{y_i}\right)^2 - \cancel{\sum_i \frac{f_i(x)^2}{y_i^2}}\right] - (n-2)\sigma_1(D(f)(x))\sigma_n(y)\sum_j \frac{f_j(x)}{y_j} \\
& \qquad + (n-2)\sigma_n(y)\sum_i \frac{f_i(x)^2}{y_i} - (n-2)\sigma_n(y)\left(\sigma_1(D(f)(x))\sum_i \frac{f_i(x)}{y_i} - \sum_i \frac{f_i(x)^2}{y_i}\right) \\
& \qquad + \sigma_1(D(f)(x))^2\sigma_{n-1}(y) - \sigma_{n-1}(y)\cancel{\sum_i f_i(x)^2} + (n-3)^2 \sigma_n(y) \sum_i \frac{f_i(x)^2}{y_i} \\
& \qquad - \left(\sigma_1(y)\sigma_{n-1}(y) \sum_i \frac{f_i(x)^2}{y_i} - \sigma_1(y)\sigma_n(y)\cancel{\sum_i \frac{f_i(x)^2}{y_i^2}} - \sigma_{n-1}(y)\cancel{\sum_i f_i(x)^2} + \sigma_n(y) \sum_i \frac{f_i(x)^2}{y_i}\right) \Bigg] \\
&= (-2)^{n-3} \Bigg[ \sigma_n(y)\sigma_1(y)\left(\sum_i \frac{f_i(x)}{y_i}\right)^2 - 2(n-2)\left(\sum_i \frac{f_i(x)}{y_i}\right)\sigma_n(y)\sigma_1(D(f)(x)) + \sigma_{n-1}(y)\sigma_1(D(f)(x))^2 \\
& \qquad + \left((n-2)^2\sigma_n(y) - \sigma_1(y)\sigma_{n-1}(y)\right)\sum_i \frac{f_i(x)^2}{y_i} \Bigg]\\
&= (-2)^{n-3} \Bigg[\frac{\sigma_n(y)}{\sigma_1(y)}\left(\sigma_1(y)\sum_i \frac{f_i(x)}{y_i} - (n-2)\sigma_1(D(f)(x))\right)^2 \\
& \qquad + \left((n-2)^2\sigma_n(y) - \sigma_1(y)\sigma_{n-1}(y)\right)\left[\sum_i \frac{f_i^2}{y_i} - \frac{\sigma_1(D(f)(x))^2}{\sigma_1(y)}\right]\Bigg].
\end{align*}

Adding $\det H$, we get
\begin{align*}
\det H(f^\alpha)(x) &\equiv f(x)^n \det H(x) + (\alpha - 1)f(x)^{n-1} \sum_{i=1}^n \sum_{j=1}^n f_i(x) f_j(x) (-1)^{i+j} \bar{H}_{ij}(x) \\
&= f(x)^n (-2)^{n-2}\left[(n-2)^2\sigma_n(y) - \sigma_1(y)\sigma_{n-1}(y)\right] \\
& \qquad + (\alpha-1) f(x)^{n-1}(-2)^{n-3} \Bigg[\frac{\sigma_n(y)}{\sigma_1(y)}\left(\sigma_1(y)\sum_i \frac{f_i(x)}{y_i} - (n-2)\sigma_1(D(f)(x))\right)^2 \\
& \qquad + \left((n-2)^2\sigma_n(y) - \sigma_1(y)\sigma_{n-1}(y)\right)\left[\sum_i \frac{f_i(x)^2}{y_i} - \frac{\sigma_1(D(f)(x))^2}{\sigma_1(y)}\right]\Bigg] \\
&= (-2f(x))^{n-1} \frac{\alpha-1}{4} \Bigg[\frac{\sigma_n(y)}{\sigma_1(y)}\left(\sigma_1(y)\sum_i \frac{f_i(x)}{y_i} - (n-2)\sigma_1(D(f)(x))\right)^2 \\
& \qquad + \left((n-2)^2\sigma_n(y) - \sigma_1(y)\sigma_{n-1}(y)\right)\left[\sum_i \frac{f_i(x)^2}{y_i} - \frac{\sigma_1(D(f)(x))^2}{\sigma_1(y)} - \frac{2}{\alpha-1}f(x) \right]\Bigg],
\end{align*}
and
\begin{align*}
\det H(f^{1/3})(x) &\equiv -\frac{1}{6}(-2f(x))^{n-1} \Bigg[\frac{\sigma_n(y)}{\sigma_1(y)}\left(\sum_j y_j\sum_i \frac{f_i(x)}{y_i} - (n-2)\sum_i f_i(x) \right)^2 \\
& \qquad + \left((n-2)^2\sigma_n(y) - \sigma_1(y)\sigma_{n-1}(y)\right)\left[\sum_i \frac{f_i(x)^2}{y_i} - \frac{\left(\sum_jf_j(x)\right)^2}{\sum_i y_i} + 3 f(x) \right]\Bigg].
\end{align*}
\end{proof}

\subsection*{\texorpdfstring{$p=n$}{p=n}}
An other computable determinant is the case $p=n$, and all the $a_k = 0$, $1 < k < n$.

We use the following lemma:
\begin{lem} \label{sigman}
Let $\mu, \nu \in \Rb$, and
\[\Lambda = \begin{pmatrix} \lambda_1 & 0 & \ldots & 0 \\ 0 & \lambda_2 & & \vdots \\ \vdots & & \ddots & 0 \\ 0 & \ldots & 0 & \lambda_n \end{pmatrix} \in \Mc_n(\Rb), \qquad U = \begin{pmatrix} u_1 \\ \vdots \\ u_n \end{pmatrix} \in \Rb^n, \qquad V = \begin{pmatrix} v_1 \\ \vdots \\ v_n \end{pmatrix} \in \Rb^n.\]
Then
\begin{align*}
\det \left(\Lambda + \mu U \otimes U + \nu V \otimes V\right) = (\det \Lambda) \Big[1 &+ \mu ^tU \Lambda^{-1}U + \nu ^tV \Lambda^{-1}V \\
&+ \mu\nu\left(\left(^tU \Lambda^{-1}U\right)\left(^tV \Lambda^{-1}V\right) - \left(^tU\Lambda^{-1}V\right)^2\right)\Big].
\end{align*}
\end{lem}
\begin{proof}
We just expand the determinant:
\begin{align*}
\det \left[\Lambda + \mu U \otimes U + \nu V \otimes V\right]
	= &\det \left[\lambda_1e_1 + \mu u_1 U + \nu v_1 V \ | \ \ldots \ | \ \lambda_ne_n + \mu u_n U + \nu v_n V\right] \\
	= &\det \left[\lambda_1e_1 \ | \ \ldots \ | \ \lambda_ne_n\right] \\
	&+ \mu\sum_{j=1}^n\det\left[\lambda_1e_1 \ | \ \ldots \ | \ u_j U \ | \ \ldots \ | \ \lambda_ne_n\right] \\
	&+ \nu\sum_{j=1}^n\det\left[\lambda_1e_1 \ | \ \ldots \ | \ v_j V \ | \ \ldots \ | \ \lambda_ne_n\right] \\
	&+ \mu\nu\sum_{i \neq j}\det\left[\lambda_1e_1 \ | \ \ldots \ | \ u_i U \ | \ \ldots \ | \ v_j V \ | \ \ldots \ | \ \lambda_ne_n\right] \\
	= &\lambda_1 \ldots \lambda_n\left[1 + \mu \sum_j \frac{u_j^2}{\lambda_j} + \nu \sum_j \frac{v_j^2}{\lambda_j} + \mu\nu\sum_{i \neq j} \frac{u_i^2v_j^2 - u_iu_jv_iv_j}{\lambda_i\lambda_j} \right].
\end{align*}
\end{proof}

\begin{prop}
Let $n \geq 3$ and $\vec{a} = (a_0, a_1, a_n) \in (\Rb_+)^3$. We set $f := f_{\vec{a}} = a_0 + a_1 \sigma_1 + a_n \sigma_n \in \Rb_+[X_1,\ldots,X_n]$. Then
\[\det H(f^{1/n})(x) \equiv \frac{(-1)^n}{\sigma_n(x)^2} \left[\left(\left(\sum_iz_i\right)^2 - (n-1)\sum_iz_i^2\right) - nf(x)a_n\sigma_n(x)\right],\]
with
\[z_i = a_1x_i + a_n\sigma_n.\]
\end{prop}
\begin{proof}
We have
\[\det H(f^{1/n}) \equiv \det\left[ fH + \left(\frac{1}{n} - 1 \right) D \otimes D\right]\]
with
\begin{align*}
H(x) &:= H(f)(x) = a_n\sigma_n(x)\begin{pmatrix} 0 & 1/x_1x_2 & \ldots & 1/x_1x_n \\ 1/x_2x_1 & 0 & & \vdots \\ \vdots & & \ddots & 1/x_{n-1}x_n \\ 1/x_nx_1 & \ldots & 1/x_nx_{n-1} & 0 \end{pmatrix} = a_n\sigma_n(x)\left[V \otimes V - T\right], \\
V &:= \begin{pmatrix} 1/x_1 \\ \vdots \\ 1/x_n \end{pmatrix}, \qquad T := \begin{pmatrix} 1/x_1^2 & & 0 \\ & \ddots & \\ 0 & & 1/x_n^2 \end{pmatrix}, \\
D(x) &:= D(f)(x) = \begin{pmatrix} a_1 + a_n \sigma_n(x)/x_1 \\ \vdots \\ a_1 + a_n \sigma_n(x)/x_n \end{pmatrix} = a_1 \Ib + a_n \sigma_n(x) V.
\end{align*}
Hence
\begin{align*}
\det H(f^{1/n})(x) &\equiv \det\left[ f(x) a_n\sigma_n(x)\left(V \otimes V - T\right) + \left(\frac{1}{n} - 1 \right) D \otimes D\right] \\
	&\equiv (-1)^n \det\left[T + \mu D \otimes D - V \otimes V\right]
\end{align*}
with $\mu = \frac{1 - 1/n}{f(x)a_n\sigma_n(x)}$. We use lemma \ref{sigman} and compute
\begin{align*}
\det\left[T + \mu D \otimes D - V \otimes V\right] &= \det T \Big[1 + \mu ^tD T^{-1}D - ^t\!V T^{-1}V \\
&- \mu\left(\left(^tD T^{-1}D\right)\left(^tV \Lambda^{-1}V\right) - \left(^tDT^{-1}V\right)^2\right)\Big] \\
&= \det T\left[1 + \mu\sum_i z_i^2 - n - \mu\left(n\sum_i z_i^2 - \left(\sum_i z_i\right)^2\right)\right] \\
&= (n-1)\det T\left[\frac{1}{nf(x)a_n\sigma_n(x)}\left(\left(\sum_i z_i\right)^2 - (n-1)\sum_i z_i^2\right) - 1\right].
\end{align*}
\end{proof}

\begin{rem}
Actually we do not need this determinant to learn something about the concavity of $f_{\vec{a}}^{1/n}$. Indeed, its restriction to $\Delta$ is never concave, so neither is $f_{\vec{a}}^{1/n}$ on $\Gamma_n$.
\[\bar{f}_{\vec{a}} = a_0 + na_1 X + a_n X^n,\]
so
\[\left(\bar{f}_{\vec{a}}^{1/n}\right)'' \equiv n(n-1)\left[-a_1^2 + a_0a_2 X^{n-2} + (n-1)a_1a_2X^{n-1}\right],\]
which changes of sign between $0$ and $+\infty$.

Moreover, one can state that $\bar{f}_{\vec{a}}$ is compatible with the conjecture \ref{conj3}: $\left(\bar{f}_{\vec{a}}^{1/n}\right)' = na_1 + na_n X^{n-1}$ is never real-rooted for $n \geq 3$, so neither is $\bar{f}_{\vec{a}}^{1/n}$ according to Rolle's lemma.
\end{rem}

%\nocite{*} %Pour faire apparatre les références non citées explicitement.
\bibliographystyle{plain}
\bibliography{../../Biblio}

\end{document}